\documentclass[11pt,oneside]{article}

\usepackage{amsthm,amsmath,amssymb,amsfonts}
\usepackage{mathtools,xfrac}
\usepackage{url,setspace,color}
\usepackage{subfigure,graphicx,epsfig,tikz,caption}
\usepackage{enumerate}
\usepackage{pdfsync,array}
\usepackage[normalem]{ulem}
\usepackage{algorithm,algorithmic}

\usepackage{geometry,verbatim}
\usepackage{makeidx,latexsym}

\usepackage[colorlinks=true]{hyperref}
\usepackage{authblk}
\usepackage[in]{fullpage}
\makeatletter
\def\imod#1{\allowbreak\mkern10mu({\operator@font mod}\,\,#1)}
\makeatother
\usepackage[compact]{titlesec}

\theoremstyle{plain}
\newtheorem{theorem}{Theorem}

\newtheorem{lemma}{Lemma}
\newtheorem*{lemma*}{Lemma}
\newtheorem{proposition}{Proposition}
\newtheorem*{proposition*}{Proposition}

\newtheorem{observation}{Observation}
\newtheorem*{observation*}{Observation}
\newtheorem{fact}{Fact}
\newtheorem*{fact*}{Fact}

\theoremstyle{definition}

\newtheorem{assumption}{Assumption}

\theoremstyle{remark}
\newtheorem{remark}{Remark}
\newtheorem*{remark*}{Remark}

\newtheorem*{example*}{Example}

\definecolor{gold}{rgb}{0.85,0.65,0}

\newcommand{\be}{\begin{eqnarray}}
\newcommand{\ee}[1]{\label{#1}\end{eqnarray}}
\newcommand{\ese}{\end{eqnarray*}}
\newcommand{\bse}{\begin{eqnarray*}}
\def\beq{\begin{equation}}
\def\eeq{\end{equation}}

\def\fnote#1{\footnote}
\newcommand{\epr}{\hfill\hbox{\hskip 4pt \vrule width 5pt height 6pt depth 1.5pt}\vspace{0.0cm}\par}

\def\ra{\rangle}
\def\la{\langle}

\newcommand{\grad}{\ensuremath{\nabla}}

\def\E{{\mathbb{E}}}

\def\N{{\mathbb{N}}}

\def\R{{\mathbb{R}}}

\def\cZ{{\cal Z}}

\newcommand{\CZ}{\mathcal{Z}}

\DeclareMathOperator{\Opt}{Opt}

\DeclareMathOperator*{\argmin}{arg\,min}
\DeclareMathOperator*{\argmax}{arg\,max}

\DeclareMathOperator{\SadVal}{SV}
\DeclareMathOperator{\SV}{SV}
\DeclareMathOperator{\sad}{sad}

\DeclareMathOperator{\Prox}{Prox}

\def\epsilonsad{\epsilon_{\sad}}

\newcommand{\dgf}{d.g.f.}

\def\log{\mathop{{\rm log}}}

\begin{document}

\title{Exploiting Problem Structure in Optimization under Uncertainty via Online Convex Optimization}
\author[1]{Nam Ho-Nguyen}
\author[1]{Fatma K{\i}l{\i}n\c{c}-Karzan}
\affil[1]{Tepper School of Business, Carnegie Mellon University, Pittsburgh, PA, 15213, USA.}
\date{August 1, 2016; last revised March 30, 2018}
\maketitle

\begin{abstract}
In this paper, we consider two paradigms that are developed to account for uncertainty in optimization models: robust optimization (RO) and joint estimation-optimization (JEO). We examine recent developments on efficient and scalable iterative first-order methods for these problems, and  show that these iterative methods can be viewed through the lens of online convex optimization (OCO). The standard OCO framework has seen much success for its ability to handle decision-making in dynamic, uncertain, and even adversarial environments. Nevertheless, our applications of interest present further flexibility in OCO via three simple modifications to standard OCO assumptions: we introduce two new concepts of \emph{weighted regret} and \emph{online saddle point problems} and study the possibility of making \emph{lookahead} (anticipatory) decisions. Our analyses demonstrate that these flexibilities introduced into the OCO framework have significant consequences whenever they are applicable. 
For example, in the strongly convex case, minimizing unweighted regret has a proven optimal bound of $O(\log(T)/T)$, whereas we show that a bound of $O(1/T)$ is possible when we consider weighted regret. Similarly, for the smooth case, considering $1$-lookahead decisions results in a $O(1/T)$ bound, compared to $O(1/\sqrt{T})$ in the standard OCO setting. Consequently, these OCO tools are instrumental in exploiting structural properties of functions and results in improved convergence rates for RO and JEO. In certain cases, our results for RO and JEO match the best known or optimal rates in the corresponding problem classes without data uncertainty.

\end{abstract}

\maketitle

\section{Introduction}\label{sec:intro}

We consider the following convex optimization problem with given input data $u = [u^1;\ldots;u^m]$:
\begin{equation}\label{eqn:uncertain-data-problem}
\min_x \left\{ f(x) :~ f^i(x,u^i) \leq 0,\ \forall i =1,\ldots,m,~~ x \in X \right\},
\end{equation}
where $X$ is a convex domain, and $f, f^1,\ldots,f^m$ are all convex functions of $x\in X$. Often, the data $u$ defining the problem~\eqref{eqn:uncertain-data-problem} are uncertain or misspecified (only approximations of the true data are available). In many applications, optimization with noisy data can have a large negative effect on performance. As an example, in portfolio optimization, the covariance matrix is difficult to estimate, and the mean-variance model is notoriously sensitive to data perturbations \cite{GoldfarbIyengar2003}. To address this, several methodologies have been developed to handle the uncertainty or misspecification of data in \eqref{eqn:uncertain-data-problem}. In this paper, we consider iterative solution methods for two different paradigms  with  
tractable models that handle uncertainty, namely,   
robust optimization and joint estimation-optimization problems, and establish a deeper connection between such iterative approaches for these   problems 
and online convex optimization. 

Robust optimization (RO) addresses data uncertainty in \eqref{eqn:uncertain-data-problem} by seeking a solution $x \in X$ that is feasible for all data realizations $u^i$ from a fixed uncertainty set $U^i$ for each constraint $f^i$, $i=1,\ldots,m$. More specifically, convex RO seeks to solve
\begin{equation}\label{eqn:RO-problem}
\min_x \left\{ f(x) : \sup_{u^i \in U^i} f^i(x,u^i) \leq 0,\ i =1,\ldots,m,\quad x \in X \right\}.
\end{equation} 
RO has been extensively studied in the literature, and we refer the reader to the paper by Ben-Tal and Nemirovski \cite{BenTalNem1998}, the book by Ben-Tal et al.\@ \cite{BenTalelGhaouiNemirovski2009} and surveys \cite{BenTalNemirovski2002,BenTalNemirovski2008,BertsimasBrownCaramanis2011,CaramanisMannorXu2011} for a detailed account of RO theory and its numerous applications.

The traditional solution method for RO is based on reformulating it first into an equivalent deterministic robust counterpart problem via duality theory, and then solving the robust counterpart as a deterministic convex optimization problem. However, the robust counterpart approach often leads to larger and much less scalable problems than the associated nominal problem of \eqref{eqn:RO-problem} where the uncertain data (noise) $[u^1;\ldots;u^m]$ is fixed to a given value. For example, it is well-known that the robust counterpart of a second-order cone program with ellipsoidal uncertainty is a semidefinite program. 
Recently, there have been several interesting developments on \emph{iterative methods} for solving RO problems that bypass the robust counterpart approach, see e.g.,  \cite{MutapcicBoyd2009,BenTalHazan2015,Ho-NguyenKK2016RO}.  
These methods solve \eqref{eqn:RO-problem} by iteratively updating the solution $x$ and the noise $[u^1;\ldots;u^m]$ to approximate their optimum  values.

Joint estimation-optimization (JEO) considers the setting where we only have uncertainty $u$ in the objective $f(x,u)$, and that the `correct' data value $u^*$ may be learned through a distinct learning process, i.e., it is characterized as a solution to a separate optimization problem $\min_u \left\{ g(u) : u \in U \right\}$. More precisely, JEO aims to solve
\begin{align}
&\min_x \left\{ f(x, u^*) :~ x \in X \right\}\tag{$\Opt(u^*)$}\label{eqn:JEO-problem}\\
\text{where} \quad & u^* \in \argmin_u \left\{ g(u) :~ u \in U \right\}. \tag{Est}\label{eqn:JEO-estimation-problem}
\end{align}
In many practical situations, JEO is solved via a \emph{sequential} method: first minimize $g(u)$ to find $u^*$, then minimize $f(x,u^*)$ to solve the problem. However, we often cannot solve for $u^*$ exactly, but instead must settle for an approximation $\bar{u} \approx u^*$.
With such a strategy, under mild Lipschitz continuity assumptions, the accuracy of \eqref{eqn:JEO-problem} is controlled by the norm of $\|\bar{u}-u^*\|$. Nevertheless, this creates the following `inconsistency' problem: when minimizing $f(x,\bar{u})$, we create a sequence of points $x_t \in X$, $t \geq 1$, which converge to the minimum of $f(x,\bar{u})$; however, the sequence will not converge to the desired minimum \eqref{eqn:JEO-problem}, and in fact will only be within $O(\|\bar{u} - u^*\|)$ accuracy. That is, this approach cannot provide asymptotically accurate solutions $x_t$. It is possible to achieve consistency via a na\"{i}ve scheme by creating a sequence of approximations $u_t$ such that $\|u_t - u^*\| \to 0$, and for each $u_t$, minimizing $f(x,u_t)$ up to accuracy $O(\|u_t - u^*\|)$ to obtain $x_t$. Then the sequence $x_t$ will be consistent, 
i.e., $\lim_{t\to\infty} f(x_t,u_t)$ converges to the optimum value of \eqref{eqn:JEO-problem}.
This na\"{i}ve scheme comes with two disadvantages: each step $t$ involves solving a complete minimization problem up to some accuracy, and furthermore the accuracy must improve at each new step. The main problem is that at each step $t$, the information from the previous steps cannot be utilized, hence they are essentially wasted. To address this, Jiang and Shanbhag \cite{JiangShanbhag2013,JiangShanbhag2014} and Ahmadi and Shanbhag \cite{AhmadiShanbhag2014} propose a scheme that \emph{jointly} solves the estimation and optimization problems, which we refer to as JEO. With this scheme, they can efficiently generate a sequence of points $x_t$ and $u_t$ such that $f(x_t,u_t)$ will indeed converge to the desired minimum \eqref{eqn:JEO-problem}, and give corresponding non-asymptotic error rates. In particular, their scheme can exploit previous information in a principled manner by ensuring that the effort in each step consists only of first-order updates.

The iterative RO methods of \cite{MutapcicBoyd2009,BenTalHazan2015,Ho-NguyenKK2016RO} and the simultaneous JEO approach of \cite{JiangShanbhag2013,JiangShanbhag2014,AhmadiShanbhag2014} both build a solution $\bar{x}$ in very similar ways: iteratively generate a \emph{solution} sequence $x_t$ and a \emph{data} sequence $u_t$ (for $t \geq 1$) that approximate the `ideal' solution and data points respectively, then perform averaging after a finite number of iterations $T$ to build an approximate solution $\bar{x}$.
A key feature in both approaches is that generating the next \emph{solution} point $x_t$ uses information from the \emph{data} sequence $u_1,\ldots,u_{t-1}$ up to iteration $t-1$, and vice versa. This intricacy is handled via tools from \emph{online convex optimization} (OCO) in the case of RO in \cite{Ho-NguyenKK2016RO}; we will demonstrate later that  the simultaneous approach of JEO can also be viewed through the lens of OCO.

OCO is part of the broader online learning (or sequential prediction) framework, which was introduced as a method to optimize decisions in a dynamic environment where the objective is changing at every time period, and at each time period we are allowed to adapt to our changing environment based on accumulated information. The origin of the online learning model can be traced back to the work of Robbins \cite{Robbins1950} on compound statistical decision problems. This framework has found a diverse set of applications in many fields; for further details see \cite{CesaBianchiLugosi2006,Hazan2016,Shalev-Shwarz2011}.

In standard OCO, we are given a convex domain $X$ and a finite time horizon $T$. In each time period $t=1,\ldots,T$, an online player chooses a decision $x_t\in X$ based on \emph{past} information from time steps $1,\ldots,t-1$ only. Then, a convex loss function $f_t:X \to \R$ is revealed, and the player suffers loss $f_t(x_t)$ and gets some feedback typically in the form of first-order information $\grad f_t(x_t)$. We call this restriction on the player \emph{non-anticipatory}, since the player cannot anticipate the next loss $f_t$ ahead of deciding $x_t$.\footnote{This is also referred to as a \emph{0-lookahead} framework.} 
In addition, it is usually assumed that the functions $f_t$ are set in advance---possibly by an all-powerful adversary that has full knowledge of our learning algorithm---and we know of only the general class of these functions. As such, it is unreasonable to compare the loss of the player across the time horizon to the best possible loss, which would require full knowledge of $f_t$ in advance of choosing $x_t$. Instead, the player's sequence of decisions $x_t$ is evaluated against the best fixed decision in hindsight, and the (average) difference is defined to be the \emph{regret}:
\begin{equation}\label{eqn:regret}
\frac{1}{T} \sum_{t=1}^{T} f_t(x_t) - \inf_{x\in X} \frac{1}{T} \sum_{t=1}^{T} f_t(x).
\end{equation}
The goal in OCO is to design efficient \emph{regret minimizing} algorithms that generate the points $x_t$ so that the regret tends to zero as $T$ increases. Therefore, in OCO we seek non-anticipatory algorithms to choose $x_t$ that guarantee
\[
\frac{1}{T} \sum_{t=1}^{T} f_t(x_t) - \inf_{x\in X} \frac{1}{T} \sum_{t=1}^{T} f_t(x) \leq r(T),\quad \lim_{T \to \infty} r(T) = 0,
\]
and the performance of our algorithms is measured by how quickly $r(T)$  
tends to $0$. While regret may seem like a weak evaluation metric, the fact that regret minimizing algorithms exist for \emph{any} sequence of functions $f_t$ is quite powerful. In particular, it allows us to handle the intricacies of simultaneously generating $x_t$ and $u_t$.

In this paper, we view iterative approaches to both RO and JEO in a \emph{unified} manner through the lens of OCO, and  demonstrate how such a view opens up the possibility of introducing simple flexibilities to standard OCO that are instrumental in exploiting structural properties of the problems and results in improved convergence rates for RO and JEO. 
Our nonstandard yet flexible OCO framework is obtained through three simple modifications to the standard OCO assumptions. These modifications are as follows:
\begin{enumerate}[(i)]
\item We introduce the concept of \emph{weighted regret}, where instead of taking uniform averages with weights $\theta_t={1/T}$ in \eqref{eqn:regret}, we are allowed to use \emph{nonuniform} weighted averages. From a modeling perspective, this allows us to capture situations where decisions $x_t$ at different time steps $t$ have varying importance.

\item We introduce the \emph{online saddle point (SP) problem}, where at each step we receive a convex-concave function $\phi_t(x,y)$ and must choose $x$ and $y$. This is an extension of the well-studied offline convex-concave SP problem, and can be thought of as a dynamic zero-sum two-player game where at each step the players are restricted to make only one move.

\item We explore the implications of \emph{1-lookahead} or \emph{anticipatory} decisions, where the learner can receive \underline{limited information} on the function $f_t$ before making the decision $x_t$. This is in contrast to most OCO settings where the learner must choose $x_t$ before any information on $f_t$ is revealed.

\end{enumerate}

Under this new OCO framework with flexibilities, we present and discuss algorithms accompanied with new regret bounds that can be better than the standard OCO ones when favorable problem structure is present. 
Our algorithms are based on online adaptations of two commonly used offline first-order methods (FOMs) from convex optimization, namely Mirror Descent and Mirror Prox. We present our developments in the flexible proximal setup of Juditsky and Nemirovski \cite{JuditNem2012Pt1,JuditNem2012Pt2} which can be further customized to the geometry of the domains. 

Our analyses demonstrate that these flexibilities introduced into the OCO framework have significant consequences whenever they are applicable. 
For example, in the strongly convex case, minimizing unweighted regret has a proven optimal bound of $O(\log(T)/T)$, whereas we show that a bound of $O(1/T)$ is possible when we consider weighted regret. Similarly, for the smooth case, considering $1$-lookahead decisions results in a $O(1/T)$ bound, compared to $O(1/\sqrt{T})$ in the standard OCO setting (see Remarks~\ref{rem:stronglyConvDiscussion} and \ref{rem:smoothLBdiscuss}).

Consequently, these new regret bounds are pivotal in exploiting structural properties of functions to achieve improved convergence rates for both RO and JEO. 
For example, in the case of RO, we demonstrate that it is possible to achieve a convergence rate of $O(1/T)$, improving over the standard $O(1/\sqrt{T})$ rate, when the functions $f^i$ satisfy certain strong convexity (or smoothness) assumptions.
These new developments then allow us to partially resolve  an open question from \cite{BenTalHazan2015} on the complexity lower bounds for solving RO via iterative techniques. 
For JEO, in addition to covering the standard setups from \cite{AhmadiShanbhag2014} in a unified manner and extending them to  the more general proximal setup, we explore a setting which was not covered in the work of \cite{JiangShanbhag2013,JiangShanbhag2014,AhmadiShanbhag2014}, when $f$ is non-smooth and strongly convex. In this setting, we provide an improved convergence rate of $O(1/T)$, which is the optimal rate even if we had the correct data $u^*$ upfront.

\subsubsection*{Related Work}

For the RO problem \eqref{eqn:RO-problem}, Mutapcic and Boyd \cite{MutapcicBoyd2009} analyzed an iterative cutting-plane-type approach, which has an exponential-in-dimension convergence guarantee of $\left( 1 + O(1/\epsilon) \right)^n$ iterations to obtain an $\epsilon$-optimal solution. Ben-Tal et al. \cite{BenTalHazan2015} suggests an approach using online convex optimization, which guarantees convergence in $O(1/\epsilon^2)$ iterations. Each iteration of \cite{MutapcicBoyd2009} and \cite{BenTalHazan2015} requires solving at least a nominal version of \eqref{eqn:RO-problem}, which can be expensive. The recent work of \cite{Ho-NguyenKK2016RO} provides a unifying framework for both approaches \cite{MutapcicBoyd2009} and \cite{BenTalHazan2015} via OCO, and presents a refined analysis which allows for a significant reduction in the computational effort of each iteration to simple first-order updates only, while enjoying a convergence guarantee of $O(1/\epsilon^2 \log(1/\epsilon))$. This reduction in the per-iteration computational cost in the approach of \cite{Ho-NguyenKK2016RO} is enough to offset the extra $\log(1/\epsilon)$ factor in the overall number of iterations; see \cite[Section 4.4]{Ho-NguyenKK2016RO} for a detailed discussion. In this paper, we examine the OCO-based framework of \cite{Ho-NguyenKK2016RO}, and provide improved convergence results under structural assumptions on the properties of functions $f^i$ for this framework.

Jiang and Shanbhag \cite{JiangShanbhag2013,JiangShanbhag2014} introduced and studied the JEO problem \eqref{eqn:JEO-problem}-\eqref{eqn:JEO-estimation-problem} in a stochastic setting, and Ahmadi and Shanbhag \cite{AhmadiShanbhag2014} examined the deterministic case. In this paper, we consider the deterministic JEO problem, for which \cite{AhmadiShanbhag2014} provided some remarkable convergence results. Specifically, they analyze the setting when $g$ is strongly convex and both $f$ and $g$ are smooth. In \cite[Proposition 3]{AhmadiShanbhag2014}, when $f$ is also strongly convex, a gradient descent-type algorithm is given with error bound of $O(T \beta^T)$ after $T$ iterations, for some $0 < \beta < 1$. In \cite[Proposition 4]{AhmadiShanbhag2014}, when $f$ is only convex, the same algorithm (with different tuning parameters) ensures an error bound of $O(1/T)$. Furthermore, when $f$ does not enjoy strong convexity or smoothness, \cite[Proposition 6]{AhmadiShanbhag2014} provides an error bound of $O(1/\sqrt{T})$. These results demonstrate that, despite access to only estimates of the true data with increasing accuracy, the simultaneous first-order JEO approach of \cite{AhmadiShanbhag2014} can achieve error bounds which are asymptotically as good or almost as good as first-order methods equipped with exact data.
Similar to RO, we show that the JEO problem can be viewed through the lens of OCO, and explore possible improvements through our flexible OCO framework.

To our knowledge, the concept of weighted regret in OCO has not been studied before. However, modification of aggregation weights as a means to speed up convergence has been explored in the stochastic optimization setting under strong convexity assumptions; see \cite{HazanKale2014,LacosteJSchmidtBach2012,NedicLee2014,RakhlinShamirSridharan2012}. Our work can be seen as an extension of these results to the adversarial setting, and in fact, one of our results, Theorem~\ref{thm:OCO-strongly-convex}, is a simple generalization of a result from \cite{LacosteJSchmidtBach2012}. Nevertheless, by stating the result in the general adversarial setting of OCO, we are able to apply it to RO and JEO, which do not fit within the stochastic optimization framework.

Mahdavi et al. \cite{MahdaviJinYang2012} introduce a special case of online SP problems to handle difficult constraints in OCO problems. The difficult constraints $s^i(x) \leq 0$ are embedded into each loss function $f_t(x)$ by aggregation with Lagrange dual multipliers $y$, to form a new loss function $\phi_t(x,y) = f_t(x) + \sum_{i=1}^m y^{(i)} s^i(x)$, which is convex in $x$ and concave in $y$. Both primal and dual variables $x,y$ are then updated each time step to obtain bounds on the regret and the violation $\sum_{t=1}^T s^i(x)$. The papers \cite{KoppelJakubiecRibeiro2015,Jenatton_etal2016} also use similar duality ideas for handling difficult constraints and objectives in online settings. Nevertheless, the convergence rates given in these papers are the usual $O(1/\sqrt{T})$ or slower. In this paper, we analyze online SP problems more generally, and explore faster rates in the $1$-lookahead setting.

Online settings with 1-lookahead naturally arise in metrical task systems \cite{BorodinLinialSaks1992,BuchbinderCNO2012,AndrewBLLMRW2013} and online display advertising \cite{Jenatton_etal2016}. In these settings, the variation of the decisions $x_1,\ldots,x_T$ across the time horizon is also penalized, and the performance of the sequence is measured as the competitive ratio of the realized loss with the best possible loss \cite{BorodinLinialSaks1992,BuchbinderCNO2012,AndrewBLLMRW2013} or as a dynamic regret term \cite{Jenatton_etal2016}. Both competitive ratio and dynamic regret objectives do not fit to our framework. Moreover, \cite[Section 4]{AndrewBLLMRW2013} show that standard regret and competitive ratio cannot be simultaneously optimized.

From an algorithmic point of view,
Mahdavi et al. \cite{MahdaviJinYang2012}, Chiang et al. \cite{ChiangYLMLJZ2012} and Yang et al. \cite{YangMahdaviJinZhou2012} examine online variants of the Mirror Prox algorithm when it is limited to work with \emph{only past information}. In particular, \cite{ChiangYLMLJZ2012,YangMahdaviJinZhou2012} provide regret bounds with `gradual variation' terms, which capture how quickly the sequence of functions $f_t$ vary. 
Rakhlin and Sridharan \cite{RakhlinSridharan2013a,RakhlinSridharan2013b} analyze $1$-lookahead decisions in OCO through the lens of \emph{predictable sequences}. They explore how one can exploit information from a single sequence $M_1,\ldots,M_T$ in an online framework, where each term $M_t$ is revealed to the player \emph{prior} to choosing the decision $x_t$. They provide the Optimistic Mirror Descent algorithm, which is essentially a generalization of Mirror Prox \cite{Nemirovski2005}, to exploit the sequence $M_1,\ldots,M_T$. In \cite{RakhlinSridharan2013a,RakhlinSridharan2013b}, they focus on uncoupled dynamics and zero-sum games, whereas our work focuses on  more general and flexible OCO problems, and designing and applying proper generalizations of FOMs such as Mirror Prox to more flexible OCO problems arising in the context of coupled optimization problems. That said, our work in Section~\ref{sec:smooth} is related to exploiting a specific predictable sequence; we elaborate on this in Remark~\ref{rem:predictable-sequence}.

\subsubsection*{Outline}\label{sec:outline}

In Section~\ref{sec:standardOCO}, we derive the concepts of weighted regret and online SP problems via the notion of affine regret, thereby allowing us to approach both problems through a common algorithmic framework, which we describe in Section~\ref{sec:weighted-regret-algorithms}. After introducing the basic proximal setup in Section~\ref{sec:prox-setup}, we analyze weighted regret OCO and online SP problems via the online Mirror Descent algorithm, and derive the standard $O(1/\sqrt{T})$ convergence rates in Section~\ref{sec:non-smooth}. We also show how strong convexity assumptions on the loss functions allow us to improve this to $O(1/T)$. In Section~\ref{sec:smooth}, we introduce and analyze an online variant of the Mirror Prox algorithm that  achieves $O(1/T)$ convergence rates under $1$-lookahead and smoothness assumptions. In Sections~\ref{sec:RO-app}~and~\ref{sec:JEO-app}, we apply the developments of Sections~\ref{sec:standardOCO}~and~\ref{sec:weighted-regret-algorithms} to the RO and JEO problems respectively. We close with a summary of our results and some future directions in Section~\ref{sec:Conclusions}.

\subsubsection*{Notation}\label{sec:prelim}

For a positive integer $n\in\N$,  we let $[n]=\{1,\ldots,n\}$ and define $\Delta_n:=\{x\in\R^n_+:~\sum_{i\in[n]} x_i=1\}$ to be the standard simplex. 
Throughout the paper,  
the subscript, e.g., $x_t,y_t,z_t,f_t,\phi_t$, is used to attribute items to the $t$-th time period or iteration.  
We use the notation $\{x_t\}_{t=1}^T$ to denote the collection of items $\{x_1,\ldots,x_T\}$. Given a vector $x\in\R^n$, we let $x^{(j)}$ denote its $j$-th coordinate for $j \in [n]$. One exception we make to this notation is that we always denote the convex combination weights $\theta\in\Delta_T$ with $\theta_t$. 
We use Matlab notation for vectors and matrices, i.e., $[x;y]$ denotes the concatenation of two column vectors $x$, $y$. 
Given $x,y\in\R^n$, $\la x,y\ra$ corresponds to the usual inner product of $x$ and $y$.  
Given a norm $\|\cdot\|$, we let $\|\cdot\|_*$ denote the corresponding dual norm. 
For $x\in\R^n$, $\|x\|_2$ denotes the Euclidean $\ell_2$-norm of $x$ defined as $\|x\|_2 = \sqrt{\la x,x\ra}$.  
We let $\partial f(x)$ be the subdifferential of $f$ taken at $x$. 
We abuse notation slightly by denoting $\grad f(x)$ for both the gradient of function $f$ at $x$ if $f$ is differentiable and a subgradient of $f$ at $x$, even if $f$ is not differentiable. If $\phi$ is of the form $\phi(x,y)$, then $\grad_x \phi(x,y)$ denotes the subgradient of $\phi$ at $x$ while keeping the other variables fixed at $y$.

\section{Generalized Regret in Online Convex Optimization}\label{sec:standardOCO}

In this section, we examine a number of generalizations of the regret concept and show how they can all be unified via an affine regret concept. We start with \emph{affine regret} given by 
\begin{equation}\label{eqn:linear-regret} 
\sum_{t=1}^{T} \la \xi_t, x_t \ra - \inf_{x \in X} \sum_{t=1}^{T} \la \xi_t, x \ra = \sup_{x \in X} \sum_{t=1}^{T} \la \xi_t, x_t - x \ra,
\end{equation}
where $\xi_t$ is a given loss vector at time $t$. Suppose that when the player makes a decision $x_t \in X$, the adversary returns $\xi_t = \grad f_t(x_t)$, where $f_t:X \to \R$ is some convex function. Then by the subgradient inequality we have $f_t(x_t) - f_t(x) \leq \la \grad f_t(x_t), x_t - x \ra = \la \xi_t, x_t - x \ra$  and hence 
\[
\frac{1}{T} \sum_{t=1}^{T} f_t(x_t) - \inf_{x \in X} \frac{1}{T} \sum_{t=1}^{T} f_t(x) \leq \sup_{x \in X} \frac{1}{T} \sum_{t=1}^{T} \la \xi_t, x_t - x \ra.
\]
This implies that the standard regret in OCO is upper bounded by the affine regret \eqref{eqn:linear-regret} where the loss vectors $\xi_t$ are the subgradients $\grad f_t(x_t)$. Then, to minimize usual regret,  it is enough to minimize the affine regret. That said, as will be discussed in Section~\ref{sec:weighted-regret-algorithms}, in order to obtain improved rates of convergence, we must go beyond affine regret and exploit further structural properties of the functions $f_t$. Even then, all the bounds from Section~\ref{sec:weighted-regret-algorithms} involve upper bounding the affine regret \eqref{eqn:linear-regret} in some fashion.

\subsection{OCO with Weighted Regret}\label{sec:W-regret}

The first flexibility we introduce to the OCO framework is scaling each time step $t$ by weights $\theta_t > 0$. With $\xi_t = \grad f_t(x_t)$,
and applying the subgradient inequality, this results in
\begin{equation}\label{eqn:weighted-regret-linear-regret-bound}
\sum_{t=1}^{T} \theta_t f_t(x_t) - \inf_{x\in X} \sum_{t=1}^{T} \theta_t f_t(x) \leq \sup_{x \in X} \sum_{t=1}^{T} \theta_t \la \xi_t, x_t - x \ra.
\end{equation}
We define the left hand side of this inequality to be the \emph{weighted regret}. From a modeling perspective, weighted regret enables us to model situations where later decisions $x_t$ carry higher importance by placing higher weights $\theta_t$ on subsequent periods $t$ (or vice versa).
For example, in a repeated game where performance of a player is aggregated from the loss at each stage, we may want to weigh the later stages more heavily than the earlier stages, since earlier stages might be used to explore the opponents' strategy, whereas in later stages we expect the player to have converged to a (near)-optimal strategy.

On the practical side, weighted regret lets us choose weights $\theta_t$ to speed up convergence. In particular, when the functions $f_t$ are strongly convex, our bounds of $O(1/T)$ for weighted regret improve on the optimal regret bounds $O(\log(T)/T)$ for the uniform weight case. Furthermore, weighted regret bounds become important when solving RO problems, where we combine two different regret terms \emph{which must have the same weights $\theta$} to obtain bounds. We discuss these practical aspects fully in Sections~\ref{sec:weighted-regret-algorithms} and \ref{sec:RO-app}.

Finally, note that while it is possible to view weighted regret as a rescaling of the functions $f_t$ with weights $\theta_t$, such a view will inevitably change the parameters associated with functions $f_t$ such as strong convexity. In contrast,  working with the weighted regret concept circumvents this issue; see Section~\ref{sec:strongly-convex} for our study on exploiting strong convexity.

Because we are interested in taking a weighted average, henceforth we will assume that we have convex combination weights $\theta:=(\theta_1,\ldots,\theta_T)\in\Delta_T$. Thus, we seek OCO algorithms for selecting $x_t$ that minimize the weighted regret and guarantee
\begin{equation}\label{eqn:OCO-weighted-regret}
\sum_{t=1}^{T} \theta_t f_t(x_t) - \inf_{x\in X} \sum_{t=1}^{T} \theta_t f_t(x) \leq r(T), \quad \lim_{T \to \infty} r(T) = 0.
\end{equation}

Regret bounds for online convex optimization algorithms naturally result in optimality gap bounds for the corresponding offline problems. 
\begin{remark}\label{rem:offline}
When the functions $f_t$ remain the same throughout the time horizon, i.e., $f_t=f$ for all $t\in[T]$, and $\bar{x}$ is taken to be the weighted sum of $\{x_t\}_{t=1}^T$ with weights $\theta\in\Delta_T$, the weighted regret in \eqref{eqn:OCO-weighted-regret} naturally bounds the standard optimality gap of solution $\bar{x}$ in the associated offline convex minimization problem $\min_{x\in X} f(x)$.
\epr
\end{remark}

\subsection{Online Saddle Point Problems}\label{sec:SPOCO}

The standard convex-concave saddle point (SP) problem is defined as
\begin{equation}\label{eqn:offline-saddlept}
\SV = \inf_{x \in X} \sup_{y \in Y} \phi(x,y) = \sup_{y \in Y} \inf_{x \in X} \phi(x,y),
\end{equation}
where $X,Y$ are nonempty compact convex sets in Euclidean spaces $\E_x,~\E_y$ and the function $\phi(x,y)$ is convex in $x$ and concave in $y$. 
Note that the latter equality in \eqref{eqn:offline-saddlept} holds because of the minimax theorem (see \cite{Sion1958}) under assumptions of compactness and convexity of the sets $X$ and $Y$, and $\phi$ admitting a convex-concave structure. 

Any convex-concave SP problem \eqref{eqn:offline-saddlept} gives rise to two convex optimization problems that are dual to each other:
\begin{equation*}\label{neq1}
\begin{array}{rclcr}
\Opt(P)&=&\inf_{x\in X}[ \overline{\phi}(x):=\sup_{y\in Y} \phi(x,y)]&&(P)\\
\Opt(D)&=&\sup_{y\in Y}[ \underline{\phi}(y):=\inf_{x\in X} \phi(x,y)] &&(D)\\
\end{array}
\end{equation*}
with $\Opt(P)=\Opt(D)=\SadVal$.
SP problem~\eqref{eqn:offline-saddlept} also leads to a monotone \emph{variational inequality} (VI) problem on $\cZ=X\times Y$: 
\[
\hbox{find  $z_*\in \cZ$ such that\ }\langle F(z),z-z_*\rangle \geq0\;\;\mbox{for all}\; z\in \cZ,
\]
 where $F:\cZ\mapsto \E_x\times \E_y$ is the monotone gradient operator given  by
\[ F(x,y) = [\grad_x \phi(x,y); - \grad_y \phi(x,y)]. \] 
It is well-known that the solutions to \eqref{eqn:offline-saddlept}---the saddle points of $\phi$ on $X\times Y$---are exactly the pairs $[x;y]$ formed by   
optimal solutions to the problems $(P)$ and $(D)$. 
They are also exactly the solutions to the associated VI problem.  

We quantify the accuracy of a candidate solution $[\bar{x},\bar{y}]$ to SP problem \eqref{eqn:offline-saddlept} with the \emph{saddle point gap} given by 
\begin{equation}\label{eqn:SPgap}
\epsilonsad^\phi(\bar{x},\bar{y}):=\overline{\phi}(\bar{x})-\underline{\phi}(\bar{y}) 
=\underbrace{\left[\overline{\phi}(\bar{x})-\Opt(P)\right]}_{\geq0}+
\underbrace{\left[\Opt(D)-\underline{\phi}(\bar{y})\right]}_{\geq0}.
\end{equation}
In order to solve \eqref{eqn:offline-saddlept} to accuracy $\epsilon > 0$, we must find $[x^\epsilon;y^\epsilon]$ such that the SP gap $\epsilonsad^\phi(x^\epsilon,y^\epsilon)\leq \epsilon$, i.e., it is small.

When $\phi(x,y)$ is convex in $x$, so is the function $\overline{\phi}(x)=\sup_{y\in Y} \phi(x,y)$. Hence, \eqref{eqn:offline-saddlept} has the interpretation of simply minimizing a convex function $\overline{\phi}(x)$ over the domain $X$. However, taking the supremum over $y\in Y$ in $\overline{\phi}(\cdot)$ may destroy some important structural properties of $\phi(x,y)$ such as smoothness. The main motivation for designing specific FOMs to solve offline SP problems in \cite{Nesterov2005,Nemirovski2005} is to exploit such structural properties of $\phi$ via the monotone gradient operator $F$ and rather not work with $\overline{\phi}(x)$ explicitly.

A natural extension of convex-concave SP problems to an online setup is as follows: We are given domains $X,Y$ and a time horizon $T$. At each time period $t\in[T]$, we  simultaneously select $[x_t;y_t] \in X \times Y$ and learn $\phi_t(x_t,y_t)$ based on a convex-concave function $\phi_t(x,y)$ revealed at the time period. We can think of this as a dynamic two-player zero-sum game, where at each stage $t$, each player makes only one move (decision) $x_t \in X$ and $y_t \in Y$ as opposed to reaching to an approximate  equilibrium. Then the goal of each player is to minimize their weighted regrets given the sequence of moves of the other player, i.e., 
\[ \sum_{t=1}^T \theta_t \phi_t(x_t,y_t) - \inf_{x \in X} \sum_{t=1}^{T} \theta_t \phi_t(x,y_t) \quad \text{and} \quad \sup_{y \in Y} \sum_{t=1}^T \theta_t \phi_t(x_t,y) - \sum_{t=1}^{T} \theta_t \phi_t(x_t,y_t).\footnote{Note that the $y$-player receives a \emph{concave reward} $\phi_t(x_t,y_t)$ at each time step, so their regret is written with the supremum.} \]

In this setup, we assume that at each period $t$, the decisions and actions (queries made to the function $\phi_t$) of each player, i.e., $x_t$ etc., are revealed to the other and vice versa \emph{immediately after} they make their decision or action. This revealed information from period $t$ can then be used by both players in their subsequent decisions and actions in the same period $t$ or in future rounds $t+1$ and so on.

Let us now examine the affine regret associated with the monotone gradient operators of the functions $\phi_t$, denoted by $F_t(x,y) = [\grad_x \phi_t(x,y); -\grad_y \phi_t(x,y)]$. More precisely, let $z = [x;y]$ and $z_t = [x_t;y_t]$, and define $\xi_t = F_t(z_t)$. Then we have the following relation on the affine regret: 
\begin{align*}
&\sup_{z \in X \times Y} \sum_{t=1}^{T} \theta_t \la \xi_t, z_t - z \ra \\
&\qquad = \sup_{z \in X \times Y} \sum_{t=1}^{T} \theta_t \left(\la \grad_x \phi_t(x_t,y_t), x_t - x \ra + \la \grad_y \phi_t(x_t,y_t), y - y_t \ra \right)\\
&\qquad = \sup_{x \in X} \sum_{t=1}^{T} \theta_t \la \grad_x \phi_t(x_t,y_t), x_t - x \ra + \sup_{y \in Y} \sum_{t=1}^{T} \theta_t \la \grad_y \phi_t(x_t,y_t), y - y_t \ra\\
&\qquad \geq \sup_{x \in X} \sum_{t=1}^{T} \theta_t (\phi_t(x_t,y_t) - \phi_t(x,y_t)) + \sup_{y \in Y} \sum_{t=1}^{T} \theta_t (\phi_t(x_t,y) - \phi_t(x_t,y_t))\\
&\qquad = \sum_{t=1}^{T} \theta_t \phi_t(x_t,y_t) - \inf_{x \in X} \sum_{t=1}^{T} \theta_t \phi_t(x,y_t) + \sup_{y \in Y} \sum_{t=1}^{T} \theta_t \phi_t(x_t,y) - \sum_{t=1}^{T} \theta_t \phi_t(x_t,y_t)\\
&\qquad = \sup_{y \in Y} \sum_{t=1}^{T} \theta_t \phi_t(x_t,y) - \inf_{x \in X} \sum_{t=1}^{T} \theta_t \phi_t(x,y_t),
\end{align*}
where the inequality follows from the convex-concave structure of $\phi_t(x,y)$ and the subgradient inequalities. Notice that the last line is simply the sum of both players' weighted regrets. Hence, minimizing the (weighted) affine regret of the gradient operators $F_t$ results in minimizing the average social loss, i.e., the sum of the players' regrets. We refer to this sum as the \emph{weighted online SP gap}, and call the problem of minimizing the weighted online SP gap the \emph{online SP problem}. More precisely, the online SP gap problem seeks OCO algorithms to generate $[x_t;y_t]$ that minimize the weighted online SP gap 
\begin{equation}\label{eqn:onlineSP-gap}
\sup_{y \in Y} \sum_{t=1}^{T} \theta_t \phi_t(x_t,y) - \inf_{x \in X} \sum_{t=1}^{T} \theta_t \phi_t(x,y_t) \leq r(T), \quad \lim_{T \to \infty} r(T) = 0.
\end{equation}

When the functions $\phi_t$ remain the same throughout the time horizon, i.e., $\phi_t=\phi$ for all $t\in[T]$, and $\bar{x}$, $\bar{y}$ are taken to be the weighted sums of $\{x_t\}_{t=1}^T$, $\{y_t\}_{t=1}^T$ respectively, the weighted online SP gap naturally bounds the standard SP gap for the underlying offline SP problem, i.e., $\epsilonsad^\phi(\bar{x},\bar{y})$ in \eqref{eqn:SPgap}.

An offline (online) SP problem can be solved by solving two related OCO problems, which can also be interpreted as two regret-minimizing players playing a static (dynamic) zero-sum game. Note that the reverse is not true in general: solving an offline (online) SP problem does not in general give us bounds on the individual regrets of each player.

The online SP gap interpretation of \eqref{eqn:onlineSP-gap} is advantageous when we relax the non-anticipatory restriction. In an online setup where 1-lookahead decisions are allowed, by examining specialized algorithms for minimizing the weighted online SP gap \eqref{eqn:onlineSP-gap} rather than employing two separate regret-minimization algorithms for the players, we can exploit both the fact that our choices $[x_t;y_t]$ may utilize the current function $\phi_t$ and any favorable structural properties of the functions $\phi_t$ such as smoothness. In Section~\ref{sec:smooth}, we introduce algorithms that minimize the weighted online SP gap \eqref{eqn:onlineSP-gap} directly. 
Our analysis demonstrates that exploiting favorable structural properties of functions $\phi_t$ plays a crucial role for obtaining better convergence rates for \eqref{eqn:onlineSP-gap}. See also Remark~\ref{rem:smoothLBdiscuss}.

\section{An Algorithmic Framework for Online Convex Optimization}\label{sec:weighted-regret-algorithms}

Many OCO algorithms are closely related to offline iterative FOMs. In this section, we first introduce some notation and key concepts related to the proximal setup for FOMs along with general properties of two classical FOMs, namely the \emph{Mirror Descent} and \emph{Mirror Prox} algorithms, that are crucial in our analysis for OCO. We then analyze the general versions of these FOMs to develop upper bounds on the weighted regret and weighted online SP gap. We follow the presentation and notation of the excellent survey \cite{JuditNem2012Pt1,JuditNem2012Pt2}. 

\subsection{Proximal Setup for the Domains}\label{sec:prox-setup}

Most {FOM}s capable of solving OCO and online SP problems are quite flexible in terms of adjusting to the geometry of the problem characterized by its domain $\cZ$. In the case of SP problems, the domain is given by $\cZ=X\times Y$ where \eqref{eqn:offline-saddlept} lives.  
The following components are standard in forming the setup for such {FOM}s and their convergence analysis:
\begin{itemize}
\item \emph{Norm}: $\|\cdot\|$ on the Euclidean space $\E$ where the domain $\cZ$ lives, along with its dual norm $\|\zeta\|_*:=\max\limits_{\|z\|\leq1}\langle\zeta,z\rangle$.
\item \emph{Distance-Generating Function} ({\dgf}): A function $\omega(z):\cZ\rightarrow \R$, which is convex and continuous on $\cZ$, and admits a selection of subgradients  $\grad\omega(z)$ that is continuous on the set $\cZ^\circ:=\{z\in \cZ:\partial\omega(z)\neq\emptyset\}$ (here $\partial\omega(z)$ is a subdifferential of $\omega$ taken at $z$), and is strongly convex with modulus 1 with respect to 
$\|\cdot\|$:
   \[ 
    \forall z',z''\in \cZ^\circ:~ \langle \grad\omega(z')-\grad\omega(z''),\ z'-z''\rangle \geq\|z'-z''\|^2. 
    \] 
\item \emph{Bregman distance}: $V_z(z'):=\omega(z')-\omega(z)-\langle \grad\omega(z),z'-z\rangle$ for all $z\in \cZ^\circ$ and $u\in \cZ$.

Note that $V_z(z') \geq \frac{1}{2} \|z-z'\|^2 \geq 0$ for all $z\in\cZ^\circ$ and $z'\in\cZ$ follows from the strong convexity of $\omega$. 

\item  \emph{Prox-mapping}: Given a \emph{prox center} $z\in \cZ^\circ$,
  \[
  \Prox_z(\xi):=\argmin\limits_{z' \in \cZ}\left\{\langle \xi,z'\rangle + V_z(z')\right\}: \E\to \cZ^\circ.
  \]

When the {\dgf} is taken as the squared $\ell_2$-norm, the prox mapping becomes the usual projection operation of the vector $z-\xi$ onto $\cZ$.
\item \emph{$\omega$-center}: $z_\omega:=\argmin\limits_{z\in \cZ}\omega(z)$.
\item \emph{Set width}: $\Omega=\Omega_z:=\max\limits_{z\in \cZ}V_{z_\omega}(z)\leq\max\limits_{z\in \cZ}\omega(z)-\min\limits_{z\in \cZ}\omega(z)$.
\end{itemize}
For common domains $\cZ$ such as simplex, Euclidean ball, and spectahedron, standard proximal setups, i.e., selection of norm $\|\cdot\|$, {\dgf} $\omega(\cdot)$, the resulting $\Prox$ computations and set widths $\Omega$ are discussed in \cite[Section 5.7]{JuditNem2012Pt1}. 
 
When we have a decomposable domain $\cZ = X \times Y$, we can build a proximal setup for $\cZ$ from the individual proximal setups on $X$ and $Y$. Given a norm $\|\cdot\|_x$ and a {\dgf} $\omega_x(\cdot)$ for the domain $X$, similarly $\|\cdot\|_y$, $\omega_y(\cdot)$ for the domain $Y$, and 
two scalars $\beta_x,\beta_y>0$, we build the {\dgf} $\omega(z)$ and $\omega$-center $z_\omega$ for $\cZ=X\times Y$ as
\[
\omega(z)=\beta_x\omega_x(x)+\beta_y\omega_y(y) \quad \mbox{ and } \quad
z_\omega = [x_{\omega_x}; y_{\omega_y}],
\]
where $\omega_x(\cdot)$ and $\omega_y(\cdot)$ as well as $x_{\omega_x}$ and $y_{\omega_y}$ are customized based on the geometry of the domains $X$ and $Y$.  In this construction, the flexibility in determining the scalars $\beta_x,\beta_y>0$ is useful in optimizing the overall convergence rate. Moreover, by letting $\xi=[\xi_x;\xi_y]$ and $z=[x;y]$, the prox mapping becomes decomposable as
\[
\Prox_z(\xi)= \left[\Prox_{x}^{\omega_x}\left({\xi_x\over\beta_x}\right);\,\Prox_{y}^{\omega_y}\left({\xi_y\over\beta_y}\right)\right],
\]
where $\Prox_{x}^{\omega_x}(\cdot)$ and $\Prox_{y}^{\omega_y}(\cdot)$ are respectively prox mappings with respect to 
$\omega_x(\cdot)$ in domain $X$ and $\omega_y(\cdot)$ in domain $Y$. 
We refer the reader to the references \cite[Section 5.7.2]{JuditNem2012Pt1} and \cite[Section 6.3.3]{JuditNem2012Pt2} for further details on how to optimally choose the parameters $\beta_x,\beta_y$ for SP problems.

\subsection{Non-Smooth Convex Functions}\label{sec:non-smooth}

In the most basic setup, our functions $f_t$ (resp.\@ $\phi_t$) are convex (resp.\@ convex-concave) and non-smooth. In this case, we analyze a generalization of Mirror Descent, outlined in Algorithm~\ref{alg:md} for bounding the weighted regret and weighted online SP gap.
\begin{algorithm}
\caption{Generalized Mirror Descent}\label{alg:md}
\begin{algorithmic}
\STATE {\bf input:~} $\omega$-center $z_\omega$, time horizon $T$, positive step sizes $\{\gamma_t\}_{t=1}^T$, and a sequence $\{\xi_t\}_{t=1}^T$.
\STATE {\bf output:~} sequence $\{z_t\}_{t=1}^T$.
\STATE $z_1 := z_\omega$.
  \FOR{$t=1,\ldots,T$}
      \STATE $z_{t+1} = \Prox_{z_{t}} (\gamma_t \xi_t)$ \label{alg:md:first_prox}
  \ENDFOR
\end{algorithmic}
\end{algorithm}

\begin{remark}\label{rem:MDnon-anticipatory}
In Algorithm~\ref{alg:md}, computation of $z_t$ depends on only $z_{t-1}$ and $\xi_{t-1}$. In the following we will examine Algorithm~\ref{alg:md} by allowing $\xi_{t-1}$ to depend on only the past information on functions $f_1,\ldots,f_{t-1}$ (or $\phi_1,\ldots,\phi_{t-1}$). Then the iterations in Algorithm~\ref{alg:md} will be based on solely the past information allowing us to carry out a \emph{non-anticipatory} analysis for  Algorithm~\ref{alg:md}.
\epr
\end{remark}

Proposition~\ref{prop:md-convergence} describes a fundamental property exhibited by the Mirror Descent updates. Its proof can be found in \cite[Proposition 5.1, Equation 5.13]{JuditNem2012Pt1}, and we include it here for completeness.

\begin{proposition}\label{prop:md-convergence}
Suppose that the sequence of vectors $\{z_t\}_{t=1}^T$ is generated by Algorithm~\ref{alg:md} for a given sequence of vectors $\{\xi_t\}_{t=1}^T$ and step sizes $\gamma_t>0$ for $t\in [T]$. Then for any $z \in \cZ$ and $t\in[T]$, we have
\begin{equation}\label{eqn:md-convergence}
\gamma_t \la \xi_t, z_t - z \ra \leq V_{z_t}(z) - V_{z_{t+1}}(z) + \frac{1}{2} \gamma_t^2 \|\xi_t\|_*^2.
\end{equation}
\end{proposition}
\begin{proof}
Recall that
\[z_{t+1} = \Prox_{z_t}(\gamma_t \xi_t) = \argmin_{z \in \CZ} \left\{ \gamma_t \la \xi_t, z \ra + V_{z_t}(z) \right\} = \argmin_{z \in \CZ} \left\{ \la \gamma_t \xi_t - \grad \omega(z_t), z \ra + \omega(z) \right\}. \]
We first prove that, for all $z \in \CZ$, $\la \gamma_t \xi_t - \grad \omega(z_t) + \grad \omega(z_{t+1}), z - z_{t+1} \ra \geq 0$. Fix some $z \in Z$ and consider the function $h_{z_{t+1},z}(s) = \la \gamma_t \xi_t - \grad \omega(z_t),  z_{t+1} + s(z - z_{t+1}) \ra + \omega(z_{t+1} + s(z - z_{t+1}))$ defined for $s \in [0,1]$. In general, $h_{z_{t+1},z}$ may not be differentiable since $\omega$ may not be, but we know that it is convex, hence subgradients exist, and by definition of $z_{t+1}$ as the minimizer, it is non-decreasing, hence all subgradients of $h$ are non-negative. In particular, all subgradients of $h_{z_{t+1},z}$ at $s=0$ are non-negative, and it is a simple exercise to check that $\la \gamma_t \xi_t - \grad \omega(z_t) + \grad \omega(z_{t+1}), z - z_{t+1} \ra$ is one such subgradient. We now know that for all $z \in \CZ$,
\[\la \gamma_t \xi_t - \grad \omega(z_t) + \grad \omega(z_{t+1}), z - z_{t+1} \ra \geq 0.\]
We thus have
\begin{align*}
\gamma_t \la \xi_t, z_t - z \ra &\leq \la \grad \omega(z_{t+1}) - \grad \omega(z_t), z - z_{t+1} \ra + \gamma_t \la \xi_t, z_t - z_{t+1} \ra\\
&= V_{z_t}(z) - V_{z_{t+1}}(z) - V_{z_t}(z_{t+1}) + \gamma_t \la \xi_t, z_t - z_{t+1} \ra\\
&\leq V_{z_t}(z) - V_{z_{t+1}}(z) - \frac{1}{2} \|z_t - z_{t+1}\|^2 + \gamma_t \la \xi_t, z_t - z_{t+1} \ra\\
&\leq V_{z_t}(z) - V_{z_{t+1}}(z) - \frac{1}{2} \|z_t - z_{t+1}\|^2 + \gamma_t \|\xi_t\|_* \|z_t - z_{t+1}\|,
\end{align*}
where the second inequality follows by strong convexity of $\omega$ and the third inequality follows by the definition of the dual norm. The result now follows by recognizing that $\max_{s} \left\{ \gamma_t \|\xi_t\|_*s - s^2/2 \right\} = \gamma_t^2 \|\xi_t\|_*^2/2$.  
\end{proof}

\subsubsection{Weighted Regret}\label{sec:basicWeightedRegret}

From Proposition~\ref{prop:md-convergence}, we may derive a bound on the weighted regret \eqref{eqn:OCO-weighted-regret} in the most general case where our functions $f_t(x)$ need only satisfy convexity and Lipschitz continuity. More precisely, we will assume the following.
\begin{assumption}\label{ass:OCO-non-smooth}
A proximal setup of Section~\ref{sec:prox-setup} exists for the domain $\cZ = X$. Each function $f_t$ is convex, and there exists $G \in(0,\infty)$ such that the subgradients of  $f_t$ are bounded, i.e., $\|\grad f_t(x)\|_* \leq G$ for all $x \in X$ and $t \in [T]$. 
\end{assumption}

\begin{theorem}\label{thm:OCO-non-smooth}
Suppose Assumption~\ref{ass:OCO-non-smooth} holds, and we are given weights $\theta \in \Delta_T$. Then running Algorithm~\ref{alg:md} with $z_t = x_t$, $\xi_t = \theta_t \grad f_t (x_t)$, and step sizes $\gamma_t =\gamma:= \sqrt{\frac{2 \Omega}{\sup_{t \in [T]} \theta_t^2 G^2 T}}$ for all $t\in[T]$ results in
\[ \sum_{t=1}^{T} \theta_t f_t(x_t) - \inf_{x \in X} \sum_{t=1}^{T} \theta_t f_t(x) \leq \sqrt{2 \Omega \left(\sup_{t \in [T]} \theta_t^2\right) G^2 T}. \]
\end{theorem}
Note that Theorem \ref{thm:OCO-non-smooth} is a simple generalization of the fundamental result of \cite{Zinkevich2003}. We include its proof for completeness.
\begin{proof}
By summing up \eqref{eqn:md-convergence} for $t\in[T]$ and writing $\gamma_t = \gamma$ as a constant we obtain
\[ \sum_{t=1}^{T} \gamma_t \la \xi_t, x_t - x \ra = \gamma \sum_{t=1}^{T} \theta_t \la \grad f_t(x_t), x_t - x \ra \leq V_{x_1}(x) - V_{x_{T+1}}(x) + \frac{\gamma^2}{2} \sum_{t=1}^{T} \theta_t^2 \|\grad f_t(x_t)\|_*^2. \]
Because $\|\theta_t \grad f_t (x_t)\|_* \leq \theta_t G \leq \left(\sup_{t \in T} \theta_t\right) G$, $V_{x_1}(x) \leq \Omega$ by our choice of $x_1$ in Algorithm~\ref{alg:md}, $-V_{x_{T+1}}(x) \leq 0$, and dividing through by $\gamma$, we reach to 
\[\sum_{t=1}^{T} \theta_t \la \grad f_t(x_t), x_t - x \ra \leq \frac{\Omega}{\gamma} + \frac{\gamma}{2} \left(\sup_{t \in [T]} \theta_t^2\right) G^2 T.\]
Optimizing the right hand side over $\gamma \geq 0$ gives us the desired upper bound of $\sqrt{2 \Omega \left(\sup_{t \in [T]} \theta_t^2\right) G^2 T}$. The left hand side of  inequality in the theorem follows from $\theta_t\geq0$ for all $t\in[T]$ and  the convexity of functions $f_t$ implying for all $x \in X$ 
\[ \la \xi_t, x_t - x \ra = \theta_t \la \grad f_t(x_t), x_t - x \ra \geq \theta_t f_t(x_t) - \theta_t f_t(x).\]
\end{proof}
The bound on weighted regret in Theorem~\ref{thm:OCO-non-smooth} is optimized when the convex combination weights $\theta \in \Delta_T$ are set to be \emph{uniform}, i.e., $\theta_t = 1/T$; in this case, the right hand side of the inequality becomes $O(1/\sqrt{T})$. 

\begin{remark}\label{rem:Omega}
We would like to highlight the importance of customizing our proximal setup based on the geometry of the domain. In many cases, weighted regret or weighted online SP gap bounds have a dependence on the set width parameter $\Omega$ associated with the proximal setup; see e.g., Theorem~\ref{thm:OCO-non-smooth}. For example, when our domain $X=\Delta_n$, equipping $X$ with a proximal setup based on negative entropy {\dgf} $\omega(x) = \sum_{j=1}^{n} x^{(j)} \log(x^{(j)})$ results in $\Omega=\log(n)$, which is almost dimension independent. Using the Euclidean {\dgf} $\omega(x) = {1\over 2}\la x,x \ra$  
on $X=\Delta_n$ leads to a suboptimal (and dimension-dependent) set width of $\Omega = \sqrt{n}$. Moreover, certain domains admit {\dgf}s that lead to quite efficient $\Prox$ computations given either in closed form or by simple computations, taking only $O(n)$ arithmetic operations. Negative entropy {\dgf} over simplex and Euclidean {\dgf} over the Euclidean unit ball are such examples. 
A possible issue for equipping the simplex with a Euclidean proximal setup is that the prox-mapping (usual projection) no longer has a closed form, but it still can be done efficiently in $O(n\log(n))$ arithmetic operations. See \cite{JuditNem2012Pt1} for a complete discussion.
\epr
\end{remark}

\subsubsection{Exploiting Strong Convexity}\label{sec:strongly-convex}

When our functions $f_t$ admit further favorable structure in the form of strong convexity, it is possible to customize Algorithm~\ref{alg:md} using specific \emph{nonuniform} weights $\theta_t$ and achieve a bound of $O(1/T)$, which is significantly better than the standard $O(1/\sqrt{T})$ bound of Theorem~\ref{thm:OCO-non-smooth} given by uniform weights. Our developments here are based on the following structural assumption.
\begin{assumption}\label{ass:OCO-strong-convex}\emph{}
\begin{itemize}
\item A proximal setup of Section~\ref{sec:prox-setup} exists for the domain $\cZ = X$.
\item The loss functions $f_t(x)$ for $t\in[T]$ have the property that the functions $f_t(x) - \alpha\, \omega(x)$ is convex for some $\alpha > 0$  independent of $t$, or equivalently 
\[ f_t(x) \leq f_t(x') + \la \grad f_t(x), x - x' \ra - \alpha V_x(x'), \quad \forall x,x' \in X, \ t \in [T]. \]
\item The subgradients of the loss functions are bounded, i.e., there exists $G\in(0,\infty)$ such that $\|\grad f_t(x)\|_* \leq G$ for all $x \in X,\ t \in [T]$. 
\end{itemize}
\end{assumption}
\begin{remark}\label{rem:strongConvexity}
When our proximal setup for $X$ is based on a Euclidean {\dgf} $\omega(x) = {1\over 2}\la x,x \ra$ and Euclidean norm $\|x\|_2$, then Assumption~\ref{ass:OCO-strong-convex} simply states that the functions $f_t$ are $\alpha$-strongly convex. In this paper, we will abuse terminology slightly and say that $f_t$ is $\alpha$-strongly convex when $f_t(x) - \alpha\, \omega(x)$ is convex, where the dependence on the {\dgf} $\omega$ will be clear from the context.
\epr
\end{remark}

In equation \eqref{eqn:weighted-regret-linear-regret-bound}, we demonstrated that the weighted regret of a sequence of functions and points $\{f_t,x_t\}_{t=1}^T$ can be upper bounded by a affine regret term with loss vectors $\xi_t = \grad f_t(x_t)$. Under Assumption~\ref{ass:OCO-strong-convex}, we can improve this upper bound via the following lemma.
\begin{lemma}\label{lemma:OCO-strong-convex}
Suppose that for $t \in [T]$, the loss functions $f_t$ satisfy Assumption~\ref{ass:OCO-strong-convex}. Given a sequence $\{x_t\}_{t=1}^T$, define $q_t(x) := \la \grad f_t(x_t), x \ra + \alpha V_{x_t}(x)$. Then the weighted regret of the sequence $\{x_t\}_{t=1}^T$ on the functions $f_t$ can be bounded by the weighted regret of the same sequence on the functions $q_t$:
\begin{align}
\sum_{t=1}^T \theta_t f_t(x_t) - \inf_{x \in X} \sum_{t=1}^T \theta_t f_t(x) 
&\leq \sup_{x \in X} \sum_{t=1}^T \theta_t \left( \la \grad f_t(x_t), x_t - x \ra - \alpha V_{x_t}(x) \right) \notag \\
&= \sum_{t=1}^T \theta_t q_t(x_t) - \inf_{x \in X} \sum_{t=1}^T \theta_t q_t(x) .
\label{eqn:weighted-regret-strongly-convex}
\end{align}
\end{lemma}
\begin{proof}
Assumption~\ref{ass:OCO-strong-convex} implies that $f_t(x) - \alpha\,  \omega(x)$ is convex, and thus the inequality holds. The equality holds since $V_{x_t}(x_t) = 0$.  
\end{proof}

Notice that since $V_{x_t}(x) \geq 0$, \eqref{eqn:weighted-regret-strongly-convex} is an improvement on the affine regret bound of \eqref{eqn:weighted-regret-linear-regret-bound}. By Lemma~\ref{lemma:OCO-strong-convex}, in order to bound the weighted regret, it suffices to bound the right hand term of \eqref{eqn:weighted-regret-strongly-convex}. By selecting the step sizes $\gamma_t$ and weights $\theta_t$ in a clever fashion, we are able to exploit the extra $-\alpha V_{x_t}(x)$ terms to improve the regret bound. This result is a generalization of the offline stochastic gradient descent algorithm equipped with a Euclidean {\dgf} based proximal setup presented in Lacoste-Julien et al.\@ \cite{LacosteJSchmidtBach2012} to the online setting with domain $X$ admitting a general proximal setup. 

\begin{theorem}\label{thm:OCO-strongly-convex}
Suppose Assumption~\ref{ass:OCO-strong-convex} holds. Fix a set of convex combination weights $\theta_t = \frac{2t}{T(T+1)}$ for $t\in[T]$. Then running  Algorithm~\ref{alg:md} with $z_t = x_t$, $\xi_t = \grad f_t(x_t)$, and step sizes $\gamma_t = \frac{2}{\alpha(t+1)}$ for all $t\in[T]$ results in
\begin{align*}
\sum_{t=1}^{T} \theta_t f_t(x_t) - \inf_{x\in X} \sum_{t=1}^{T} \theta_t f_t(x) \leq \sup_{x \in X} \sum_{t=1}^T \theta_t \left( \la \grad f_t(x_t), x_t - x \ra - \alpha V_{x_t}(x) \right) \leq \frac{2 G^2}{\alpha\,(T+1)}.
\end{align*} 
\end{theorem}
\begin{proof}
By Lemma~\ref{lemma:OCO-strong-convex}, the first inequality holds, so we focus on the second. Proposition~\ref{prop:md-convergence} gives us the following inequality for all $x \in X$
\begin{align*}
\gamma_t \la \xi_t, x_t - x \ra = \gamma_t \la \grad f_t(x_t), x_t - x \ra 
&\leq V_{x_t}(x) - V_{x_{t+1}}(x) + \frac{\gamma_t^2}{2} \| \xi_t \|_*^2 \\
&= V_{x_t}(x) - V_{x_{t+1}}(x) + \frac{\gamma_t^2}{2} \|\grad f_t(x_t) \|_*^2.
\end{align*}
This, along with $\|\grad f_t(x_t) \|_* \leq G$ implies
\begin{equation}\label{eqn:regret-strong-convexity-onestep}
\la \grad f_t(x_t), x_t - x \ra - \alpha V_{x_t}(x) \leq \frac{1}{\gamma_t} V_{x_t}(x) - \frac{1}{\gamma_t} V_{x_{t+1}}(x) - \alpha V_{x_t}(x) + \frac{\gamma_t G^2}{2}.
\end{equation}
Multiplying \eqref{eqn:regret-strong-convexity-onestep} by $\theta_t$ and summing over $t \in [T]$ establishes the inequalities below. 
\[
\sum_{t=1}^T \theta_t \left( \grad f_t(x_t), x_t - x \ra - \alpha V_{x_t}(x) \right) \leq \sum_{t=1}^{T} \theta_t\!\! \left(\!\frac{1}{\gamma_t}\! V_{x_t}(x) \!-\! \frac{1}{\gamma_t} \!V_{x_{t+1}}(x) \!-\! \alpha V_{x_t}(x) \!+\! \frac{\gamma_t G^2}{2}\!\right). \label{eqn:regret-strong-convexity-intermediate}
\]
Now, when $\gamma_t = \frac{2}{\alpha\,(t+1)}$, we arrive at
\[ 
\frac{1}{\gamma_t} V_{x_t}(x) - \frac{1}{\gamma_t} V_{x_{t+1}}(x) - \alpha V_{x_t}(x) + \frac{\gamma_t G^2}{2} 
= \frac{\alpha\,(t-1)}{2}V_{x_t}(x) - \frac{\alpha\,(t+1)}{2}V_{x_{t+1}}(x) + \frac{G^2}{\alpha\,(t+1)}. 
\]
Multiplying this by $t$ gives us
\[t\left( \frac{1}{\gamma_t} V_{x_t}(x) - \frac{1}{\gamma_t} V_{x_{t+1}}(x) - \alpha V_{x_t}(x) + \frac{\gamma_t G^2}{2} \right) \leq \frac{\alpha\,(t-1)t}{2} V_{x_t}(x) - \frac{\alpha\, t(t+1)}{2} V_{x_{t+1}}(x) + \frac{G^2}{\alpha}.\]
After summing this over $t\in[T]$ and noting that the first two terms telescope, the coefficient in front of $V_{x_1}(x)$ is zero, and $V_{x_{T+1}}(x)\geq 0$, we deduce
\[
\sum_{t=1}^{T} t\left( \frac{1}{\gamma_t} V_{x_t}(x) - \frac{1}{\gamma_t} V_{x_{t+1}}(x) - \alpha V_{x_t}(x) + \frac{\gamma_t G^2}{2} \right) \leq \frac{G^2 T}{\alpha} - \frac{\alpha\, T(T+1)}{2}  V_{x_{T+1}}(x) \leq \frac{G^2 T}{\alpha}.
\]
Dividing both sides of this inequality by ${T(T+1)\over 2}$ leads to
\[
\sum_{t=1}^{T} \theta_t \left(\frac{1}{\gamma_t} V_{x_t}(x) - \frac{1}{\gamma_t} V_{x_{t+1}}(x) - \alpha V_{(x_t)}(x) + \frac{\gamma_t G^2}{2}\right) \leq \frac{2 G^2}{\alpha\,(T+1)}, 
\]
which establishes the second inequality. 
\end{proof}

Let us revisit Remark~\ref{rem:Omega} on customizing the proximal setup based on the geometry of the domain.
\begin{remark}\label{rem:stronglyConvexOmega}
In contrast to Theorem~\ref{thm:OCO-non-smooth}, the bound of Theorem~\ref{thm:OCO-strongly-convex} has no dependence on set width $\Omega$.  Nevertheless, customization of the proximal setup, in particular selection of {\dgf} $\omega$ plays an important role in Theorem~\ref{thm:OCO-strongly-convex} through Assumption~\ref{ass:OCO-strong-convex}. In many cases, it is much more likely to encounter functions $f_t$ that are $\alpha$-strongly convex in the usual sense, i.e., $f_t(x) - \alpha \|x\|_2^2/2$ is convex, but it may not be possible to ensure the convexity of $f_t(x) - \alpha\, \omega(x)$ with respect to a different  {\dgf} $\omega$. 
In such cases, it is possible (and more desirable) to select a {\dgf} $\omega$ that will ensure that the strong convexity requirement of Assumption~\ref{ass:OCO-strong-convex} is satisfied. Because the bound of Theorem~\ref{thm:OCO-strongly-convex} has no dependence on $\Omega$, such a selection of $\omega$ will not adversely affect overall the weighted regret bound of Theorem~\ref{thm:OCO-strongly-convex}. 
\epr
\end{remark}

\begin{remark}\label{rem:stronglyConvDiscussion}
For strongly convex losses, Theorem~\ref{thm:OCO-strongly-convex} establishes an upper bound of $O(1/T)$ on weighted regret. In contrast to this, Hazan and Kale \cite{HazanKale2014} established a lower bound of $O(\log(T)/T)$ for minimizing standard regret in OCO with strongly convex loss functions. The main distinguishing feature of  \cite{HazanKale2014} and our result in Theorem~\ref{thm:OCO-strongly-convex} is that while \cite{HazanKale2014} considers the case of using uniform weights $\theta_t = 1/T$ only, we are allowed to use nonuniform (in fact increasing) weights $\theta_t = 2t/(T^2 + T)$. The faster rate of $O(1/T)$ in Theorem~\ref{thm:OCO-strongly-convex} is a result of this flexibility in our setup due to the weighted regret concept that lets us choose nonuniform weights.
\epr
\end{remark}

\subsubsection{Weighted Online SP gap}\label{sec:non-smooth-SPgap}

Algorithm~\ref{alg:md} can also be utilized in bounding the weighted online SP gap \eqref{eqn:onlineSP-gap}. In this case, in addition to a convex-concave structure assumption on functions $\phi_t(x,y)$, we assume boundedness of specific monotone gradient operators associated with $\phi_t(x,y)$.
\begin{assumption}\label{ass:OnlineSP-non-smooth}
A proximal setup of Section~\ref{sec:prox-setup} exists for the domain $\cZ = X \times Y$. Each function $\phi_t(x,y)$ is convex in $x$ and concave in $y$, and there exists $G \in(0,\infty)$ such that \sloppy $\left\|[\grad_x \phi_t(x,y); -\grad_y \phi_t(x,y) ] \right\|_* \leq G$ for all $x \in X$, $y \in Y$ and $t \in [T]$. 
\end{assumption}

\begin{theorem}\label{thm:OnlineSP-non-smooth}
Suppose Assumption~\ref{ass:OnlineSP-non-smooth} holds, and we are given convex combination weights $\theta \in \Delta_T$. 
Then running Algorithm~\ref{alg:md} with $z_t = [x_t; y_t]$, $\xi_t = \theta_t [\grad_x \phi_t(x_t,y_t); -\grad_y \phi_t(x_t,y_t) ]$, and step sizes $\gamma_t = \sqrt{\frac{2 \Omega}{\sup_{t \in [T]} \theta_t^2 G^2 T}}$ for all $t\in[T]$ gives us 
\[ \sup_{y \in Y} \sum_{t=1}^{T} \theta_t \phi_t(x_t,y) - \inf_{x \in X} \sum_{t=1}^{T} \theta_t \phi_t(x,y_t) \leq \sqrt{2 \Omega \left(\sup_{t \in [T]} \theta_t^2\right) G^2 T}. \]
\end{theorem}
\begin{proof}
The proof proceeds exactly as the proof of Theorem~\ref{thm:OCO-non-smooth} to arrive at
\[ \sum_{t=1}^{T} \la \xi_t, z_t - z \ra \leq \sqrt{2 \Omega \left(\sup_{t \in [T]} \theta_t^2\right) G^2 T}\]
for all $z = [x;y] \in X \times Y$. Then, from the convex-concave structure of the function $\phi_t$, we have for all $z = [x;y] \in X \times Y$ and all $t\in[T]$, 
\begin{align*}
\la \xi_t, z_t - z \ra &= \theta_t \la \grad_x \phi_t(x_t,y_t), x_t - x \ra + \theta_t \la \grad_y \phi_t(x_t,y_t), y - y_t \ra\\
&\geq \theta_t (\phi_t(x_t,y_t) - \phi_t(x,y_t)) + \theta_t ( \phi_t(x_t,y) - \phi_t(x_t,y_t) )\\
&= \theta_t \phi_t(x_t,y) - \theta_t \phi_t(x,y_t).
\end{align*}
The result then follows by combining the inequality above with the inequality that provides the upper bound on the term $\sum_{t=1}^T \la \xi_t, z_t - z \ra$. 
\end{proof}

\begin{remark}\label{rem:nonuniformT1}
Uniform weights $\theta_t = 1/T$ minimize $\sup_{t\in[T]} \theta_t$ and result in a regret (online SP gap) bound of $O(1/\sqrt{T})$ in Theorem~\ref{thm:OCO-non-smooth} (Theorem~\ref{thm:OnlineSP-non-smooth}). Moreover, Theorems~\ref{thm:OCO-non-smooth}~and~\ref{thm:OnlineSP-non-smooth} can accommodate a variety of convex combination weights $\theta\in\Delta_T$ via adapting their step sizes $\gamma_t$ and still achieve bounds of form $O(1/\sqrt{T})$. For example, this is the case when the nonuniform weights $\theta_t = 2t/(T^2 + T)$ from Theorem~\ref{thm:OCO-strongly-convex} are used in these results. Employing nonuniform weights becomes more consequential when we have to run several OCO or online SP algorithms in conjunction with each other using the \emph{same} weights $\theta_t$ in all of them. Such a situation arises in solving robust feasibility problems, which we discuss in Section~\ref{sec:RO-app}.
\epr
\end{remark}

\subsection{Exploiting Lookahead and Smoothness}\label{sec:smooth}

In offline convex optimization when minimizing a \emph{smooth} convex function 
over a convex domain, the Mirror Prox algorithm of \cite{Nemirovski2005} admits a better convergence rate than Mirror Descent and is thus preferable. In this section we demonstrate that the same improvement is also attainable in an online setting when our functions exhibit a smooth structure and our setting allows for \emph{1-lookahead}---that is, we are allowed to a \emph{limited} query access to our current function $f_t$ at time period $t$ before we make our decision $z_t$. In fact, we query $f_t$ only once in each period $t$.

As discussed in the Introduction, 1-lookahead setting may prevent it being applicable in certain online settings. In addition, if at iteration $t$ we are given multiple query access to $f_t$ (or $\phi_t$), we can \emph{guarantee} that the weighted regret (online SP gap) will be non-positive by directly minimizing $f_t$ (solving for the SP of $\phi_t$). However, solving a complete optimization problem at each iteration may be expensive, and hence even in the situations where we have multiple query access to $f_t$ at iteration $t$, it may be preferable to use our more efficient methods to bound the weighted regret (online SP gap). We present an example of such a situation, solving robust feasibility problems, in the next section. 

Our analysis is based on the generalization of Mirror Prox outlined in Algorithm~\ref{alg:mp}. 
\begin{algorithm}[h!t]
\caption{Generalized Mirror Prox}\label{alg:mp}
\begin{algorithmic}
\STATE {\bf input:~} $\omega$-center $z_\omega$, time horizon $T$, positive step sizes $\{\gamma_t\}_{t=1}^T$, and sequences $\{\eta_t,\xi_t\}_{t=1}^T$.
\STATE {\bf output:~} sequence $\{z_t\}_{t=1}^T$.
\STATE $v_1 := z_\omega$\;
\FOR{$t=1,\ldots,T$}
	\STATE $z_t = \Prox_{v_t} (\gamma_t \eta_t)$. \label{alg:first_prox}
	\STATE $v_{t+1} = \Prox_{v_t}  (\gamma_t \xi_t)$.  \label{alg:second_prox}
\ENDFOR
\end{algorithmic}
\end{algorithm}

Proposition~\ref{prop:mp-convergence} states a fundamental property of Mirror Prox updates which is instrumental in the derivation of our bounds. Its proof can be found in \cite[Lemma 6.2 and Proposition 6.1]{JuditNem2012Pt2}, which we reproduce here for completeness.
\begin{proposition}\label{prop:mp-convergence}
Suppose that the sequences of vectors $\{v_t,z_t\}_{t=1}^T$ are generated by Algorithm~\ref{alg:mp} for the given sequences $\{\eta_t,\xi_t\}_{t=1}^T$ and step sizes $\gamma_t>0$ for $t\in [T]$. Then for any $z \in \cZ$ and $t\in[T]$, we have
\[ \gamma_t \la \xi_t, z_t - z \ra \leq V_{v_t}(z) - V_{v_{t+1}}(z) + \frac{1}{2} \left( \gamma_t^2 \| \xi_t - \eta_t\|_*^2 - \|z_t - v_t\|^2 \right). \]
\end{proposition}
\begin{proof}
Recall that
\begin{align*}
z_t = \Prox_{v_t}(\gamma_t \eta_t) &= \argmin_{z \in \CZ} \left\{ \la \gamma_t \eta_t - \grad \omega(v_t), z \ra + \omega(z) \right\}\\
v_{t+1} = \Prox_{v_t}(\gamma_t \xi_t) &= \argmin_{z \in \CZ} \left\{ \la \gamma_t \xi_t - \grad \omega(v_t), z \ra + \omega(z) \right\}.
\end{align*}
Using the same optimality condition proved in Proposition \ref{prop:md-convergence}, we have for all $z \in \CZ$
\begin{align*}
\la \gamma_t \eta_t - \grad \omega(v_t) + \grad \omega(z_t), z - z_t \ra &\geq 0\\
\la \gamma_t \xi_t - \grad \omega(v_t) + \grad \omega(v_{t+1}), z - v_{t+1} \ra &\geq 0.
\end{align*}
Rearranging the second inequality, we see that
\begin{align*}
\gamma_t \la \xi_t, z_t - z \ra &\leq \gamma_t \la \xi_t, z_t - v_{t+1} \ra + \la \grad \omega(v_{t+1}) - \grad \omega(v_t), z - v_{t+1} \ra\\
&= \gamma_t \la \xi_t, z_t - v_{t+1} \ra + V_{v_t}(z) - V_{v_{t+1}}(z) - V_{v_t}(v_{t+1}).
\end{align*}
Substituting $z = v_{t+1}$ into the first inequality gives
\begin{align*}
\gamma_t \la \xi_t, z_t - v_{t+1} \ra &\leq \gamma_t \la \xi_t - \eta_t, z_t - v_{t+1} \ra + \la \grad \omega(z_t) - \grad \omega(v_t), v_{t+1} - z_t \ra\\
&= \gamma_t \la \xi_t - \eta_t, z_t - v_{t+1} \ra + V_{v_t}(v_{t+1}) - V_{z_t}(v_{t+1}) - V_{v_t}(z_t).
\end{align*}
Combining the previous two inequalities, we have for all $z \in \CZ$
\begin{align*}
\gamma_t \la \xi_t, z_t - z \ra &\leq \gamma_t \la \xi_t - \eta_t, z_t - v_{t+1} \ra + V_{v_t}(z) - V_{v_{t+1}}(z) - V_{z_t}(v_{t+1}) - V_{v_t}(z_t)\\
&\leq V_{v_t}(z) - V_{v_{t+1}}(z) + \gamma_t \|\xi_t - \eta_t \|_* \|z_t - v_{t+1}\| - \frac{1}{2} \|z_t - v_{t+1}\|^2 - \frac{1}{2} \|z_t - v_t\|^2,
\end{align*}
where the second inequality follows by Cauchy-Schwarz and strong convexity of $\omega$. The result now follows by recognizing that for any $s \geq 0$, $\gamma_t \|\xi_t - \eta_t \|_* s - s^2/2 \leq \gamma_t^2 \|\xi - \eta_t\|_*^2/2$.
\end{proof}

We analyze Algorithm~\ref{alg:mp} under the following smoothness assumption and derive an improved rate of convergence for minimizing weighted regret.
\begin{assumption}\label{ass:OCO-smooth}
A proximal setup of Section~\ref{sec:prox-setup} exists for the domain $\cZ = X$. Each function $f_t(x)$ is convex in $x$, and there exists $L \in(0,\infty)$ such that 
$ 
\left\| \grad f_t(x) - \grad f_t(v) \right\|_* \leq L \|x - v\| 
$ 
holds for all $x,v \in X$ and all $t \in [T]$. 
\end{assumption}
\begin{theorem}\label{thm:OCO-smooth}
Suppose Assumption~\ref{ass:OCO-smooth} holds, and we are given weights $\theta \in \Delta_T$. Then running Algorithm~\ref{alg:mp} with $z_t = x_t$, $\eta_t = \theta_t \grad f_t(v_t)$, $\xi_t = \theta_t \grad f_t(z_t)$, and step sizes $\gamma_t = {1\over\left( L \sup_{t \in T} \theta_t \right)}$ for all $t\in[T]$ leads to  
\[ \sum_{t=1}^{T} \theta_t f_t(x_t) - \inf_{x \in X} \sum_{t=1}^{T} \theta_t f_t(x) \leq \Omega L \sup_{t \in [T]} \theta_t. \]
\end{theorem}
\begin{proof}
From Assumption~\ref{ass:OCO-smooth},  we have for all $t\in[T]$
\[ \|\xi_t - \eta_t\|_* = \theta_t \|\grad f_t(x_t) - \grad f_t(v_t)\|_* \leq L\, \theta_t \|x_t - v_t\| \leq L \sup_{t \in [T]} \theta_t \|x_t - v_t\|. \]
Thus, by setting $\gamma_t = {1\over \left( L \sup_{t \in [T]} \theta_t \right)}$, we deduce $\gamma_t^2 \| \xi_t - \eta_t\|_*^2 - \|x_t - v_t\|^2 \leq 0$ for all $t\in[T]$. Then from Proposition~\ref{prop:mp-convergence} we obtain for all $x \in X$ and $t\in[T]$ 
\[ \la \xi_t, x_t - x \ra = \theta_t \la \grad f_t(x_t), x_t - x \ra \leq \left(V_{v_t}(x) - V_{v_{t+1}}(x)\right) L \sup_{t \in [T]} \theta_t. \]
Summing this inequality over $t\in [T]$ and using $V_{v_1}(x) \leq \Omega$, $V_{v_{T+1}}(x) \geq 0$, we get 
\[ \sum_{t=1}^{T} \la \xi_t, x_t - x \ra = \sum_{t=1}^{T} \theta_t \la \grad f_t(x_t), x_t - x \ra \leq \Omega L \sup_{t \in [T]} \theta_t.\]
The result then follows from convexity of $f_t$ and using the subgradient inequality $\la \grad f_t(x_t), x_t - x \ra \geq f_t(x_t) - f_t(x)$. 
\end{proof}

A similar result holds for the online SP gap under the following analogous smoothness assumption.
\begin{assumption}\label{ass:OnlineSP-smooth}
A proximal setup of Section~\ref{sec:prox-setup} exists for the domain $\cZ = X \times Y$, and we denote $z = [x;y]$. Each function $\phi_t(x,y)$ is convex in $x$ and concave in $y$. Denoting $F_t(z) = [\grad_x \phi_t(x,y); -\grad_y \phi_t(x,y) ]$, there exists $L \in(0,\infty)$ such that for all $v,z \in \cZ$ and all $t \in [T]$, we have
\[ \left\| F_t(z) - F_t(v) \right\|_* \leq L \|z - v\|. \]
\end{assumption}

\begin{remark}\label{rem:Lipschitz}
A sufficient condition for the Lipschitz continuity of monotone gradient operators $F_t$ of Assumption~\ref{ass:OnlineSP-smooth} is Lipschitz continuity of their partial subgradients. For brevity, we omit the proof of this; see  \cite{JuditNem2012Pt2,Nemirovski2005} for further details.
\epr
\end{remark}

\begin{theorem}\label{thm:OnlineSP-smooth}
Suppose Assumption~\ref{ass:OnlineSP-smooth} holds, and we are given weights $\theta \in \Delta_T$. Then running Algorithm~\ref{alg:mp} with $z_t = [x_t; y_t]$, $\eta_t = \theta_t F_t(v_t)$, $\xi_t = \theta_t F_t(z_t)$, and step sizes $\gamma_t = {1\over\left( L \sup_{t \in T} \theta_t \right)}$ for all $t\in[T]$ leads to 
\[ \sup_{y \in Y} \sum_{t=1}^{T} \theta_t \phi_t(x_t,y) - \inf_{x \in X} \sum_{t=1}^{T} \theta_t \phi_t(x,y_t) \leq \Omega L \sup_{t \in [T]} \theta_t. \]
\end{theorem}
\begin{proof}
Following the outline of the proof of Theorem~\ref{thm:OCO-smooth}, we obtain
\[ \sum_{t=1}^{T} \la \xi_t, z_t - z \ra = \sum_{t=1}^{T} \theta_t \la F_t(z_t), z_t - z \ra \leq \Omega L \sup_{t \in [T]} \theta_t \]
for all $z = [x;y] \in X \times Y$. As in the proof of Theorem~\ref{thm:OnlineSP-non-smooth}, using the convex-concave structure of the functions $\phi_t$, we arrive at 
\[\theta_t \la F_t(z_t), z_t - z \ra \geq \theta_t \phi_t(x_t,y) - \theta_t \phi_t(x,y_t),\]
which establishes the result. 
\end{proof}

\begin{remark}\label{rem:nonuniformT2}
As discussed in Remark~\ref{rem:nonuniformT1}, when the convex combination weights $\theta_t$ are set to be either uniform weights $\theta_t = 1/T$ or nonuniform weights $\theta_t = 2t/(T^2 + T)$ from Theorem~\ref{thm:OCO-strongly-convex}, we have $ \sup_{t \in [T]} \theta_t=O(1/T)$, and thus we achieve a better weighted regret (online SP gap) bound of $O(1/T)$ in Theorem~\ref{thm:OCO-smooth} (Theorem~\ref{thm:OnlineSP-smooth}) than the $O(1/\sqrt{T})$ bound of Theorem~\ref{thm:OCO-non-smooth} (Theorem~\ref{thm:OnlineSP-non-smooth}). 
\epr
\end{remark}

There is a fundamental distinction between Algorithms~\ref{alg:md}~and
~\ref{alg:mp} in terms of their anticipatory/non-anticipatory behavior. \sloppy This distinction between anticipatory/non-anticipatory behavior is important in the context of using these algorithms for coupled optimization problems. We discuss this next. 
\begin{remark}\label{rem:anticipativity}
When Algorithm~\ref{alg:mp} is utilized in Theorems~\ref{thm:OCO-smooth}~and~\ref{thm:OnlineSP-smooth}, at step $t$, in order to compute the decision $z_t = \Prox_{v_t}(\gamma_t \eta_t)$, where $v_t \in \cZ$ is a point computed in the previous step, we utilize the knowledge of the current function $f_t$ or $\phi_t$ because $\eta_t = \theta_t \grad f_t(v_t)$ or $\eta_t = \theta_t F_t(v_t)$. Therefore, Algorithm~\ref{alg:mp} is categorized as \emph{1-lookahead} or \emph{anticipatory}. This is in contrast to the non-anticipatory nature of Algorithm~\ref{alg:md} analyzed in Theorems~\ref{thm:OCO-non-smooth},~\ref{thm:OCO-strongly-convex},~and~\ref{thm:OnlineSP-non-smooth}, where computing $z_t = \Prox_{z_{t-1}}(\gamma_{t-1} \xi_{t-1})$ only required knowledge of the previous step $t-1$ because $\xi_{t-1}$ was determined based on only $\grad f_{t-1}(z_{t-1})$ or $F_{t-1}(z_{t-1})$. 
\epr
\end{remark}

\begin{remark}\label{rem:predictable-sequence}
Rakhlin and Sridharan \cite{RakhlinSridharan2013a,RakhlinSridharan2013b} also explore OCO  with anticipatory decisions through the lens of \emph{predictable sequences} $\{M_t\}_{t=1}^T$. More precisely, they also examine how regret bounds are affected when the player is allowed to utilize side information $M_t$ before choosing $x_t$ at time $t$. They propose the Optimistic Mirror Descent (OpMD) algorithm, which is a special case of Algorithm~\ref{alg:mp} for $\eta_t = M_t$, $\xi_t = \grad f_t(z_t)$ and $\theta_t=1/T$, and are able to recover the offline Mirror Prox algorithm from \cite{Nemirovski2005} for smooth offline convex optimization and smooth offline SP problems. In fact, our results in Theorem~\ref{thm:OCO-smooth} and Theorem~\ref{thm:OnlineSP-smooth} can be derived from \cite[Lemma 1]{RakhlinSridharan2013b} by specifying the predictable sequences $M_t = \theta_t \grad f_t(v_t)$ and $M_t = \theta_t F_t(v_t)$ respectively. Here, we allow the player to have access \emph{only to gradient information} of $f_t$ or $\phi_t$ at time $t$. Because the focus of \cite{RakhlinSridharan2013a,RakhlinSridharan2013b} was different, the observation that the OpMD algorithm can obtain faster $O(1/T)$ convergence rates in the $1$-lookahead setting was not made before.
\epr
\end{remark}

\begin{remark}\label{rem:smoothLBdiscuss}
It is known that the OCO regret bounds with general smooth loss  functions have a lower bound complexity of at least $O(1/\sqrt{T})$ 
(this holds even for the case of linear loss functions \cite[Theorem 5]{AbernethyBartlett2008}). This is in contrast to the faster rate of $O(1/T)$ established in Theorem~\ref{thm:OCO-smooth}. The lookahead  
nature of our analysis of Algorithm~\ref{alg:mp} discussed in Remark~\ref{rem:anticipativity} plays a crucial role for achieving the speedup established in Theorem~\ref{thm:OCO-smooth}. 
\epr
\end{remark}

\section{Application: Robust Optimization}\label{sec:RO-app}

In this section, we apply our developments on OCO to solving the robust optimization (RO) problem \eqref{eqn:RO-problem}.
Instead of solving \eqref{eqn:RO-problem} directly, we examine the associated \emph{robust feasibility problem}: given desired accuracy $\epsilon>0$,
\begin{equation}\label{eqn:RO-feas-problem}
\begin{cases}
\text{\emph{Either}: find}\ \ x \in X \quad\text{s.t.}\quad \sup_{u^i\in U^i} f^i(x,u^i) \leq \epsilon \quad \forall i\in[m];\\
\text{\emph{or}: declare infeasibility, } \forall x \in X,\ \exists i \in [m] \quad \text{s.t.} \quad \sup_{u^i\in U^i} f^i(x,u^i) > 0.
\end{cases} 
\end{equation}
We note that optimizing an objective function $f(x)$ via feasibility oracle \eqref{eqn:RO-feas-problem} will incur only an extra $\log(1/\epsilon)$ multiplicative factor in the number of iterations. This approximate $\epsilon$-feasibility problem is motivated by how most convex optimization solvers certify their solutions. We are interested in the number of iterations needed to solve \eqref{eqn:RO-feas-problem}, which will depend on the accuracy parameter $\epsilon$.

It was established in \cite{Ho-NguyenKK2016RO} that, under the basic convexity assumptions, \eqref{eqn:RO-feas-problem} can be solved by standard OCO algorithms achieving $O(1/\epsilon^2)$ convergence rate and requiring only basic arithmetic operations and subgradient computations in each iteration. 
In this section, we examine how our regret bounds from  Section~\ref{sec:weighted-regret-algorithms} can improve the $O(1/\epsilon^2)$ convergence rate for \eqref{eqn:RO-feas-problem} under certain structural assumptions on the constraint functions $f^i$. We first define some notation. We denote $u := [u^1;\ldots;u^m]$, $U = U^1 \times \ldots \times U^m$ and $Y := \Delta_m$. Given sequences $x_t \in X$, $u_t \in U$, $y_t \in Y$ for $t\in[T]$ and weights $\theta \in \Delta_T$, we define
\begin{align*}
\epsilon^\circ(\{x_t,u_t,\theta_t\}_{t=1}^T) &:= \max_{i\in[m]} \left\{\sup_{u^i\in U^i} \sum_{t=1}^{T} \theta_t f^i(x_t,u^i) - \sum_{t=1}^{T} \theta_t f^i(x_t,u_t^i)\right\},\\
\epsilon^\bullet(\{x_t,u_t,y_t,\theta_t\}_{t=1}^T) &:= \max_{i\in[m]} \sum_{t=1}^{T} \theta_t f^i(x_t,u_t^i) - \inf_{x\in X} \sum_{t=1}^{T} \theta_t \sum_{i=1}^{m} y^{(i)}_t f^i(x,u_t^i),\quad\text{and}\\
\epsilon^\bullet(\{x_t,u_t,\theta_t\}_{t=1}^T) &:= \sum_{t=1}^{T} \theta_t \max_{i \in [m]} f^i(x_t,u_t^i) - \inf_{x \in X} \sum_{t=1}^T \theta_t \max_{i \in [m]} f^i(x,u_t^i).
\end{align*}
The following results from \cite{Ho-NguyenKK2016RO} states how \eqref{eqn:RO-feas-problem} can be verified in an iterative fashion.
\begin{theorem}[{\cite[Theorem 3.2, Corollary 3.1]{Ho-NguyenKK2016RO}}]\label{thm:robust-feas-oracle}
Let $x_t\in X$, $u_t\in U$, $y_t\in\Delta_m$ for $t\in[T]$, $\theta\in\Delta_T$, and $\tau \in (0,1)$. 
If $\epsilon^\circ(\{x_t,u_t,\theta_t\}_{t=1}^T) \leq \tau \epsilon$ and $\max_{i\in[m]} \sum_{t=1}^{T} \theta_t f^i(x_t,u_t^i)\leq (1-\tau) \epsilon$, then the solution $\bar{x}_T := \sum_{t=1}^{T} \theta_t x_t$ is $\epsilon$-feasible with respect to \eqref{eqn:RO-feas-problem}. 
If $\epsilon^\bullet(\{x_t,u_t,y_t,\theta_t\}_{t=1}^T)\leq (1-\tau) \epsilon$ and $\max_{i\in[m]} \sum_{t=1}^{T} \theta_t f^i(x_t,u_t^i) > (1-\tau) \epsilon$, then \eqref{eqn:RO-feas-problem} is infeasible. When all but $\{y_t\}_{t=1}^T$ is given, there exists an appropriate choice of $y_t \in \Delta_m$ such that  $\epsilon^\bullet(\{x_t,u_t,y_t,\theta_t\}_{t=1}^T) \leq \epsilon^\bullet(\{x_t,u_t,\theta_t\}_{t=1}^T)$.
Thus, if $\epsilon^\bullet(\{x_t,u_t,\theta_t\}_{t=1}^T)\leq (1-\tau) \epsilon$ and $\max_{i\in[m]} \sum_{t=1}^{T} \theta_t f^i(x_t,u_t^i) > (1-\tau) \epsilon$, then \eqref{eqn:RO-feas-problem} is infeasible.
\end{theorem}

Thus, solving the robust feasibility problem \eqref{eqn:RO-feas-problem} reduces to bounding $\epsilon^\circ(\{x_t,u_t,\theta_t\}_{t=1}^T)$ and $\epsilon^\bullet(\{x_t,u_t,y_t,\theta_t\}_{t=1}^T)$ (or $\epsilon^\bullet(\{x_t,u_t,\theta_t\}_{t=1}^T)$), and then evaluating $\max_{i\in[m]} \sum_{t=1}^{T} \theta_t f^i(x_t,u_t^i)$.
Intuitively, robust feasibility can be seen as a two-player zero sum game, where one player chooses the $\{u_t\}_{t=1}^T$ and the other player chooses $\{x_t\}_{t=1}^T$. We can think of $\epsilon^\circ(\{x_t,u_t,\theta_t\}_{t=1}^T)$ and $\epsilon^\bullet(\{x_t,u_t,y_t,\theta_t\}_{t=1}^T)$ (or $\epsilon^\bullet(\{x_t,u_t,\theta_t\}_{t=1}^T)$) as approximations to the regret of each player.

We first discuss how to bound these terms \emph{individually}. After that, we discuss how to combine these bounds properly to solve \eqref{eqn:RO-feas-problem} by taking into account the common weights $\theta \in \Delta_T$ and any non-anticipatory/lookahead properties of the algorithms.

\begin{observation}\label{obs:RO-terms-OCO}
Given a sequence $\{x_t\}_{t=1}^T$, define the functions $f_t^i(u^i) := -f^i(x_t,u^i)$. Then the term $\epsilon^\circ(\{x_t,u_t,\theta_t\}_{t=1}^T)$ can be written as the maximum of weighted regret terms \eqref{eqn:OCO-weighted-regret} with the functions $f_t^i$ and weights $\theta \in \Delta_T$ over the sequences $\{u_t^i\}_{t=1}^T$:
\[ \epsilon^\circ(\{x_t,u_t,\theta_t\}_{t=1}^T) = \max_{i \in [m]} \left\{ \sum_{t=1}^T \theta_t f_t^i(u_t^i) - \inf_{u^i \in U^i} \sum_{t=1}^T \theta_t f_t^i(u^i) \right\}. \]
Given a sequence $\{u_t\}_{t=1}^T$, define the functions $\phi_t(x,y) := \sum_{i=1}^m y^{(i)} f^i(x,u_t^i)$. Then $\epsilon^\bullet(\{x_t,u_t,y_t,\theta_t\}_{t=1}^T)$ can be written as a weighted online saddle point gap term \eqref{eqn:SPgap} with functions $\phi_t$ and weights $\theta \in \Delta_T$ over the sequence $\{x_t,y_t\}_{t=1}^T$:
\[ \epsilon^\bullet(\{x_t,u_t,y_t,\theta_t\}_{t=1}^T) = \max_{y \in Y} \sum_{t=1}^T \theta_t \phi_t(x_t,y) - \inf_{x \in X} \sum_{t=1}^T \theta_t \phi_t(x,y_t). \]
Furthermore, let $h_t(x) := \max_{i \in [m]} f^i(x,u_t^i)$. Then $\epsilon^\bullet(\{x_t,u_t,\theta_t\}_{t=1}^T)$ can be written as a weighted regret term \eqref{eqn:OCO-weighted-regret} with functions $h_t$ and weights $\theta \in \Delta_T$ over the sequence $\{x_t\}_{t=1}^T$:
\[ \epsilon^\bullet(\{x_t,u_t,\theta_t\}_{t=1}^T) = \sum_{t=1}^T \theta_t h_t(x_t) - \inf_{x \in X} \sum_{t=1}^T \theta_t h_t(x). \]
\end{observation}
Observation~\ref{obs:RO-terms-OCO} states that we may bound the terms $\epsilon^\circ(\{x_t,u_t,\theta_t\}_{t=1}^T)$, $\epsilon^\bullet(\{x_t,u_t,y_t,\theta_t\}_{t=1}^T)$ and $\epsilon^\bullet(\{x_t,u_t,\theta_t\}_{t=1}^T)$ using OCO results from Section~\ref{sec:weighted-regret-algorithms}.

We have the following basic setup assumptions.
\begin{assumption}\label{ass:RO-prox-setup}\emph{}
\begin{itemize}
\item The domain $X$ is convex and admits a proximal setup with norm $\|\cdot\|_X$ and set width $\Omega_X$ as in Section~\ref{sec:prox-setup}.
\item For $i \in [m]$, the uncertainty sets $U^i$ are convex and admit proximal setups with norms $\|\cdot\|_{(i)}$ and set widths $\Omega_U < \infty$ as in Section~\ref{sec:prox-setup}.
\end{itemize}
\end{assumption}
\begin{assumption}\label{ass:RO-convexity-functions}
For each $i \in [m]$, the functions $f^i(x,u^i)$ are convex in $x$, concave in $u^i$, and are Lipschitz continuous in each variable, i.e., the subgradients are bounded: for all $u^i \in U^i$, $\|\grad_x f^i(x,u^i)\|_{X,*} \leq G_X < \infty$, and for all $x \in X$, $\|\grad_u f^i(x,u^i)\|_{(i),*} \leq G_U < \infty$.
\end{assumption}
Under Assumption~\ref{ass:RO-convexity-functions}, the functions $f_t^i(u^i)$ and $h_t(x)$ defined in Observation~\ref{obs:RO-terms-OCO} are convex in $u^i$ and $x$ respectively, and the functions $\phi_t(x,y)$ are convex-concave in $x$ and $y$. In \cite[Section 4.1]{Ho-NguyenKK2016RO}, it is shown that we can bound $\epsilon^\circ(\{x_t,u_t,\theta_t\}_{t=1}^T)$ and $\epsilon^\bullet(\{x_t,u_t,\theta_t\}_{t=1}^T)$ by $O(1/\sqrt{T})$, which then allows us to solve \eqref{eqn:RO-feas-problem} in $T=O(1/\epsilon^2)$ iterations. We will now examine how to improve these bounds under strong convexity and smoothness assumptions on the constraint functions $f^i$, which will then allow us to improve the rate for solving \eqref{eqn:RO-feas-problem}.

We first examine the bounds under strong convexity assumptions.
\begin{assumption}\label{ass:RO-strong-convex-u}
For each $i \in [m]$ and any fixed $x \in X$, the functions $f^i(x,u^i)$ are $\alpha_U^i$-strongly concave in $u^i$: there exists $\alpha_U^i > 0$ such that $-f^i(x,u^i) - \alpha_U^i \omega(u^i)$ is convex in $u^i$, where $\omega^i$ is the {\dgf} from the proximal setup for $U^i$. Furthermore, let $\alpha_U := \min_{i \in [m]} \alpha_U^i$.
\end{assumption}
\begin{proposition}\label{prop:RO-strong-convex-rate-weighted-regret}
Suppose that Assumptions~\ref{ass:RO-prox-setup},~\ref{ass:RO-convexity-functions}~and~\ref{ass:RO-strong-convex-u} hold. Fix any $i \in [m]$ and the set of convex combination weights $\theta_t = \frac{2t}{T(T+1)}$. For any sequence $\{x_t\}_{t=1}^T$, running Algorithm~\ref{alg:md} with {\dgf} $\omega^i$, $z_t = u_t^i$, $\xi_t = -\grad_u f^i(x_t,u_t^i)$ and $\gamma_t = \frac{2}{\alpha_U(t+1)}$ guarantees that
\[ \sup_{u^i\in U^i} \sum_{t=1}^{T} \theta_t f^i(x_t,u^i) - \sum_{t=1}^{T} \theta_t f^i(x_t,u_t^i) \leq \frac{2G_U^2}{\alpha_U(T+1)}. \]
In particular, for $\theta_t = \frac{2t}{T(T+1)}$, we can choose a sequence $\{u_t\}_{t=1}^T$ such that for any sequence $\{x_t\}_{t=1}^T$, we guarantee $\epsilon^\circ(\{x_t,u_t,\theta_t\}_{t=1}^T) \leq O(1/T)$.
\end{proposition}
\begin{proof}
Assumption~\ref{ass:OCO-strong-convex} holds since Assumptions~\ref{ass:RO-prox-setup},~\ref{ass:RO-convexity-functions}~and~\ref{ass:RO-strong-convex-u} hold. Theorem~\ref{thm:OCO-strongly-convex} then applies to obtain the upper bounds on the regret terms.
\end{proof}
\begin{assumption}\label{ass:RO-strong-convex-x}
There exists $\alpha_X^i > 0$ such that for each $i \in [m]$ and each fixed $u^i \in U^i$, the function $f^i(x,u^i)$ is $\alpha_X^i$-strongly convex in $x$, that is, $f^i(x,u^i) - \alpha_X^i \omega(x)$ is convex, where $\omega$ is the {\dgf} from the proximal setup for $X$. Furthermore, define $\alpha_X := \min_{i \in [m]} \alpha_X^i$.
\end{assumption}
\begin{proposition}\label{prop:RO-strong-convex-onlineSP}
Suppose that Assumptions~\ref{ass:RO-prox-setup},~\ref{ass:RO-convexity-functions}~and~\ref{ass:RO-strong-convex-x} hold. Fix the set of convex combination weights $\theta_t = \frac{2t}{T(T+1)}$. For any sequence $\{u_t\}_{t=1}^T$, running Algorithm~\ref{alg:md} with $z_t = x_t$, $\xi_t = \grad_x f^{i(t)}(x_t,u_t^{i(t)})$ where $i(t) = \argmax_{i \in [m]} f^i(x_t,u_t^i)$, and $\gamma_t = \frac{2}{\alpha_X(t+1)}$ guarantees that
\[\epsilon^\bullet(\{x_t,u_t,\theta_t\}_{t=1}^T) \leq \frac{2 G_X^2}{\alpha_X (T+1)} = O\left( \frac{1}{T} \right).\]
\end{proposition}
\begin{proof}
Assumption~\ref{ass:OCO-strong-convex} holds since Assumptions~\ref{ass:RO-prox-setup},~\ref{ass:RO-convexity-functions}~and~\ref{ass:RO-strong-convex-x} hold, and for any $u$, the function $h_u(x) = \max_{i \in m} f^i(x,u^i)$ is strongly convex in $x$ with parameter $\alpha_X = \min_{i \in [m]} \alpha_X^i$. Theorem~\ref{thm:OCO-strongly-convex} then applies to obtain the upper bound on the regret term.
\end{proof}

We now examine the bounds under smoothness assumptions.
\begin{assumption}\label{ass:RO-smooth-u}
For each $i \in [m]$ and any fixed $x \in X$, the functions $f^i(x,u^i)$ are $L_U$-smooth in $u^i$: there exists $L_U < \infty$ such that for any $u^i,(u^i)' \in U^i$,
\[ \| \grad_u f^i(x,u^i) - \grad_u f^i(x,(u^i)')\|_{i,*} \leq L_U \|u^i - (u^i)'\|_i.\]
\end{assumption}
\begin{proposition}\label{prop:RO-smooth-rate-weighted-regret}
Suppose that Assumptions~\ref{ass:RO-prox-setup},~\ref{ass:RO-convexity-functions}~and~\ref{ass:RO-smooth-u} hold. Fix any $i \in [m]$. For any sequence $\{x_t\}_{t=1}^T$, running Algorithm~\ref{alg:mp} with $z_t = u_t^i$, $\eta_t = -\theta_t \grad_u f^i(x_t,v_t^i)$, $\xi_t = -\theta_t \grad_u f^i(x_t,u_t^i)$ and $\gamma_t = \frac{1}{L_U \sup_{t \in [T]} \theta_t}$ guarantees that
\[ \sup_{u^i\in U^i} \sum_{t=1}^{T} \theta_t f^i(x_t,u^i) - \sum_{t=1}^{T} \theta_t f^i(x_t,u_t^i) \leq \Omega_U L_U \sup_{t \in [T]} \theta_t. \]
In particular, for uniform weights $\theta_t = 1/T$ or increasing weights $\theta_t = \frac{2t}{T(T+1)}$, we can choose a sequence $\{u_t\}_{t=1}^T$ such that for any sequence $\{x_t\}_{t=1}^T$, we guarantee $\epsilon^\circ(\{x_t,u_t,\theta_t\}_{t=1}^T) \leq O(1/T)$.
\end{proposition}
\begin{proof}
Assumption~\ref{ass:OCO-smooth} holds since Assumptions~\ref{ass:RO-prox-setup},~\ref{ass:RO-convexity-functions} and~\ref{ass:RO-smooth-u} hold. Theorem~\ref{thm:OCO-smooth} then applies to obtain the upper bounds on the regret terms.
\end{proof}
Before continuing, we note that the functions $\max_{i \in [m]} f^i(x,u^i)$ are non-smooth in $x$ in general, so we will not examine the term $\epsilon^\bullet(\{x_t,u_t,\theta_t\}_{t=1}^T)$. Instead, we examine the `smoothed' term $\epsilon^\bullet(\{x_t,u_t,y_t,\theta_t\}_{t=1}^T)$, where the functions $\sum_{i=1}^m y^{(i)} f^i(x,u^i)$ are convex-concave and smooth in $[x;y]$. That is, we will bound the online saddle point gap \eqref{eqn:onlineSP-gap} from Observation~\ref{obs:RO-terms-OCO}.
\begin{assumption}\label{ass:RO-smooth-x}
For each $i \in [m]$ and any fixed $u^i \in U^i$, the functions $f^i(x,u^i)$ are $L_X$-smooth in $x^i$: there exists $L_X < \infty$ such that for any $x,x' \in X$,
\[ \| \grad_x f^i(x,u^i) - \grad_x f^i(x',u^i)\|_{X,*} \leq L_X \|x - x'\|_X.\]
\end{assumption}
\begin{observation}\label{obs:RO-joint-domain-setup}
Our domain is now $X \times Y$, since we add the variables $y \in Y = \Delta_m$. For the simplex $Y$, there exists a proximal setup with $\ell_1$-norm and set width $\Omega_y = \log(m)$. As mentioned in Section~\ref{sec:prox-setup}, we can construct a norm and proximal setup for $X \times Y$ according to \cite[Section 5.7.2]{JuditNem2012Pt1} and \cite[Section 6.3.3]{JuditNem2012Pt2}. Then the set width of this hybrid setup is $\Omega_{X,Y} = 1$, and under Assumption~\ref{ass:RO-smooth-x}, the smoothness parameter for the function $\phi_u(x,y) = \sum_{i=1}^m y^{(i)} f^i(x,u^i)$ with the constructed norm will be
\begin{equation}\label{eqn:RO-joint-smoothness-parameter}
L_{X,Y} := L_X \Omega_X + 2 G_X \sqrt{\Omega_X \log(m)},
\end{equation}
where $G_X$ is the bound on $\|\grad_x f^i(x,u^i)\|_{X,*}$. We refer to \cite[Section 5]{Nemirovski2005} and \cite[Section 6.3.3]{JuditNem2012Pt2} for further details.
\end{observation}
\begin{proposition}\label{prop:RO-smooth-rate-onlineSP}
Suppose that Assumptions~\ref{ass:RO-prox-setup},~\ref{ass:RO-convexity-functions}~and~\ref{ass:RO-smooth-x} hold. Fix any $i \in [m]$. For any sequence $\{u_t\}_{t=1}^T$, denote $\phi_t(x,y) = \sum_{i=1}^m y^{(i)} f^i(x,u_t^i)$. Running Algorithm~\ref{alg:mp} with $z_t = [x_t;y_t]$, $\eta_t = \theta_t [\grad_x \phi_t(v_t);-\grad_y \phi_t(v_t)]$, $\xi_t = \theta_t [\grad_x \phi_t(z_t); -\grad_y \phi_t(z_t)]$ and $\gamma_t = \frac{1}{L_{X,Y} \sup_{t \in [T]} \theta_t}$ guarantees that
\[\epsilon^\bullet(\{x_t,u_t,y_t,\theta_t\}_{t=1}^T) \leq \Omega_{X,Y} L_{X,Y} \sup_{t \in [T]} \theta_t = \left( L_X \Omega_X + 2 G_X \sqrt{\Omega_X \log(m)} \right) \sup_{t \in [T]} \theta_t.\]
In particular, for uniform weights $\theta_t = 1/T$, or for increasing weights $\theta_t = \frac{2t}{T(T+1)}$, we can choose a sequence $\{u_t\}_{t=1}^T$ such that for any sequence $\{x_t\}_{t=1}^T$, we guarantee $\epsilon^\bullet(\{x_t,u_t,y_t,\theta_t\}_{t=1}^T) \leq O\left(\frac{\sqrt{\log(m)}}{T}\right)$.
\end{proposition}
\begin{proof}
Assumptions~\ref{ass:RO-prox-setup},~\ref{ass:RO-convexity-functions}~and~\ref{ass:RO-smooth-x} along with Observation~\ref{obs:RO-joint-domain-setup} imply that Assumption~\ref{ass:OnlineSP-smooth} holds. Then from Theorem~\ref{thm:OnlineSP-smooth}, we obtain the upper bounds on the regret terms.
\end{proof}

We now examine how to combine our results to solve the robust feasibility problem \eqref{eqn:RO-feas-problem}. To solve \eqref{eqn:RO-feas-problem}, we must choose weights $\theta \in \Delta_T$ and generate sequences $\{x_t,u_t,y_t\}_{t=1}^T$ to \emph{simultaneously} bound $\epsilon^\circ(\{x_t,u_t,\theta_t\}_{t=1}^T)$ and one of $\epsilon^\bullet(\{x_t,u_t,y_t,\theta_t\}_{t=1}^T)$ or $\epsilon^\bullet(\{x_t,u_t,\theta_t\}_{t=1}^T)$. Depending on the structural assumptions,  we would like to combine Propositions~\ref{prop:RO-strong-convex-rate-weighted-regret},~\ref{prop:RO-smooth-rate-weighted-regret} and Propositions~\ref{prop:RO-strong-convex-onlineSP},~\ref{prop:RO-smooth-rate-onlineSP} in a valid fashion to achieve the best possible rate. Every combination is valid, except for Propositions~\ref{prop:RO-smooth-rate-weighted-regret}~and~\ref{prop:RO-smooth-rate-onlineSP} because of the 1-lookahead (anticipatory) nature of Algorithm~\ref{alg:mp}. We discuss this below. 
\begin{remark}\label{rem:RO-anticipativity}
Note that the sequences $\{u_t\}_{t=1}^T$ and $\{x_t,y_t\}_{t=1}^T$ (or just $\{x_t\}_{t=1}^T$) are generated by two different processes which use inter-related information. Hence, we have to ensure that the information available to each process is sufficient to generate the next step. For example, suppose that we use Proposition~\ref{prop:RO-smooth-rate-weighted-regret} to generate the sequence $\{u_t\}_{t=1}^T$. By Remark~\ref{rem:anticipativity}, at iteration $t$, for each $i \in [m]$ we require the knowledge of the function $f_t^i(u^i) = -f^i(x_t,u^i)$ to compute $u_t^i$. In other words, we need $x_t$ to compute $u_t$. As a consequence, we must compute $x_t$ using only knowledge of previous iterations $\{u_s\}_{s=1}^{t-1}$. Therefore, by Remark~\ref{rem:MDnon-anticipatory}, we cannot use Proposition~\ref{prop:RO-smooth-rate-onlineSP}; only Proposition~\ref{prop:RO-strong-convex-onlineSP} can be utilized.
\epr
\end{remark}
In the light of Remark~\ref{rem:RO-anticipativity}, we can combine these  propositions in three different ways under various structural assumptions. We state three results which improve on the $O(1/\epsilon^2)$ convergence from \cite{Ho-NguyenKK2016RO}. The proofs of these are straightforward applications of the relevant propositions and hence are omitted.
\begin{theorem}\label{thm:RO-strong-both}
Suppose that Assumptions~\ref{ass:RO-prox-setup},~\ref{ass:RO-convexity-functions},~\ref{ass:RO-strong-convex-u}~and~\ref{ass:RO-strong-convex-x} hold. Then we can solve \eqref{eqn:RO-feas-problem} to within $\epsilon$-approximation in $T = O(1/\epsilon)$ iterations by employing Proposition~\ref{prop:RO-strong-convex-rate-weighted-regret} to generate $\{u_t\}_{t=1}^T$ and Proposition~\ref{prop:RO-strong-convex-onlineSP} to generate $\{x_t\}_{t=1}^T$ using increasing weights $\theta_t = \frac{2t}{T(T+1)}$. Here, both $x_t$ and $u_t$ are computed with the knowledge of only past iterates $x_{t-1},u_{t-1}$.
\end{theorem}
\begin{theorem}\label{thm:RO-strong-u-smooth-x}
Suppose that Assumptions~\ref{ass:RO-prox-setup},~\ref{ass:RO-convexity-functions},~\ref{ass:RO-strong-convex-u}~and~\ref{ass:RO-smooth-x} hold. Then we can solve \eqref{eqn:RO-feas-problem} to within $\epsilon$-approximation in $T = O(\sqrt{log(m)}/\epsilon)$ iterations by employing Proposition~\ref{prop:RO-strong-convex-rate-weighted-regret} to generate $\{u_t\}_{t=1}^T$ and Proposition~\ref{prop:RO-smooth-rate-onlineSP} to generate $\{x_t,y_t\}_{t=1}^T$ using increasing weights $\theta_t = \frac{2t}{T(T+1)}$. Here, $u_t$ is computed with knowledge of $x_{t-1},u_{t-1}$, while $[x_t;y_t]$ is computed with the knowledge of $u_{t-1},u_t$ and $[x_{t-1};y_{t-1}]$.
\end{theorem}
\begin{theorem}\label{thm:RO-smooth-u-strong-x}
Suppose that Assumptions~\ref{ass:RO-prox-setup},~\ref{ass:RO-convexity-functions},~\ref{ass:RO-smooth-u}~and~\ref{ass:RO-strong-convex-x} hold. Then we can solve \eqref{eqn:RO-feas-problem} to within $\epsilon$-approximation in $T = O(1/\epsilon)$ iterations by employing Proposition~\ref{prop:RO-smooth-rate-weighted-regret} to generate $\{u_t\}_{t=1}^T$ and Proposition~\ref{prop:RO-strong-convex-onlineSP} to generate $\{x_t\}_{t=1}^T$ using increasing weights $\theta_t = \frac{2t}{T(T+1)}$. Here, $x_t$ is computed with knowledge of $x_{t-1},u_{t-1}$, while $u_t$ is computed with the knowledge of $u_{t-1},x_{t-1},x_t$.
\end{theorem}

\begin{remark}\label{rem:RO-other-oracles}
As shown in \cite[Sections 4.2, 4.3]{Ho-NguyenKK2016RO}, OCO algorithms are not the only ways to bound the terms $\epsilon^\circ(\{x_t,u_t,\theta_t\}_{t=1}^T)$ and $\epsilon^\bullet(\{x_t,u_t,\theta_t\}_{t=1}^T)$. Instead, we can use pessimization oracles from \cite{MutapcicBoyd2009} to bound $\epsilon^\circ(\{x_t,u_t,\theta_t\}_{t=1}^T)$, or nominal feasibility oracles from \cite{BenTalHazan2015} to bound $\epsilon^\bullet(\{x_t,u_t,\theta_t\}_{t=1}^T)$. A reasonable idea is to combine these oracles with Propositions~\ref{prop:RO-strong-convex-rate-weighted-regret},~\ref{prop:RO-strong-convex-onlineSP},~\ref{prop:RO-smooth-rate-weighted-regret},~\ref{prop:RO-smooth-rate-onlineSP} to obtain improved rates. However, we meet a challenge similar to Remark~\ref{rem:RO-anticipativity}. In iteration $t$, the pessimization oracles of \cite{MutapcicBoyd2009} need knowledge of $x_t$ to compute $u_t$ (see \cite[Remark 4.1]{Ho-NguyenKK2016RO}), while the nominal feasibility oracle of \cite{BenTalHazan2015} needs knowledge of $u_t$ to compute $x_t$ (see \cite[Remark 4.2]{Ho-NguyenKK2016RO}). Therefore, only Propositions~\ref{prop:RO-strong-convex-rate-weighted-regret}~and~\ref{prop:RO-strong-convex-onlineSP} may be used to improve the oracle-based rates. 
Nevertheless, this still allows us to partially answer the following open question from \cite[Section 5]{BenTalHazan2015}: is it possible to improve the $O(1/\epsilon^2)$ oracle calls required to solve \eqref{eqn:RO-feas-problem}? Our results imply the following partial affirmative answer: if every $f^i(x,u^i)$ is \emph{strongly concave} in $u^i$, then Proposition~\ref{prop:RO-strong-convex-rate-weighted-regret} can be employed to generate $\{u_t\}_{t=1}^T$, which guarantees a solution to \eqref{eqn:RO-feas-problem} in $T=O(1/\epsilon)$ iterations. It remains open whether a provable lower bound on the number of iterations exists with or without additional favorable structure such as strong concavity.
\epr
\end{remark}

\section{Application: Joint Estimation-Optimization}\label{sec:JEO-app}

In this section, we examine the joint estimation-optimization (JEO) problems \eqref{eqn:JEO-problem}-\eqref{eqn:JEO-estimation-problem}. 
We first establish a relation between iterative methods for JEO problem and regret minimization in OCO. We then show that our results from Section~\ref{sec:weighted-regret-algorithms} can recover
most of the results from \cite{AhmadiShanbhag2014}, e.g., when $f$ is smooth or non-smooth and is not strongly convex, and immediately extend these to proximal setups. 
In addition, we cover the case when $f$ is strongly convex but non-smooth, which 
as stated in the introduction, 
is not examined in the prior literature  \cite{JiangShanbhag2013,JiangShanbhag2014,AhmadiShanbhag2014}.
We first state our basic setup assumptions on the domains and the function $f$.
\begin{assumption}\label{ass:JEO-basic-assumptions}\emph{}
\begin{itemize}
\item The domain $X$ is convex and admits a proximal setup as in Section~\ref{sec:prox-setup} with set width $\Omega$. Furthermore, it is compact, with $\max_{x,u \in X} \|x - u\| \leq D < \infty$.
\item For all $u \in U$, the function $f(\cdot,u)$ is convex in $x \in X$, and is Lipschitz continuous, i.e., the gradients $\grad_x f(x,u)$ are bounded by a constant $G_{f,X} \!>\! 0$ independent of $u$.
\end{itemize}
\end{assumption}
\begin{assumption}\label{ass:JEO-strongly-convex}
For any fixed $u \in U$, strong convexity, i.e., Assumption~\ref{ass:OCO-strong-convex} holds for any 
 $f(\cdot,u)$ with uniform strong convexity parameter $\alpha_{f,X} \!>\! 0$ independent of $u$.
\end{assumption}
\begin{assumption}\label{ass:JEO-smooth}
For any fixed $u \in U$, smoothness, i.e., Assumption~\ref{ass:OCO-smooth} holds for the function $f(\cdot,u)$ with uniform smoothness parameter $L_{f,X} \geq 0$ independent of $u$.
\end{assumption}

As in \cite{AhmadiShanbhag2014}, we also assume access to a sequence of points $\{u_t\}_{t=1}^T$ which approximate the correct data $u^*$ in \eqref{eqn:JEO-estimation-problem}. Whenever a new approximation $u_{t-1}$ is revealed, we generate a point $x_t$ based on this new data. After $T$ iterations, we build the point $\bar{x}_T = \sum_{t=1}^T \theta_t x_t \in X$ through averaging. Using this scheme, we bound the approximation quality of $\bar{x}_T$ by two terms: an affine regret term based on the sequences $\{x_t,u_t\}_{t=1}^T$ and the function $f$, and a penalty term for our inability to work with the correct data $u^*$. We start with a simple lemma, which establishes the link between JEO and OCO.

\begin{lemma}\label{lemma:JEO-regret-f-strongly-convex}
Suppose that Assumption~\ref{ass:JEO-basic-assumptions} holds. Given sequences $\{x_t,u_t\}_{t=1}^T$ and weights $\theta_t \in \Delta_T$, define $q_t(x) := \la \grad_x f(x_t,u_t), x \ra$ and $\bar{x}_T := \sum_{t=1}^T \theta_t x_t \in X$. Then 
\begin{align*} 
f(\bar{x}_T,u^*) - \min_{x \in X} f(x,u^*) \leq &~\sum_{t=1}^T \theta_t q_t(x_t) - \inf_{x \in X} \sum_{t=1}^T \theta_t q_t(x) \\
&\quad + D \sum_{t=1}^T \theta_t \| \grad_x f(x_t,u_t) - \grad_x f(x_t,u^*) \|_*. 
\end{align*}
If, in addition, Assumption~\ref{ass:JEO-strongly-convex} holds, then the same holds with $q_t(x) := \la \grad_x f(x_t,u_t), x \ra + \alpha_{f,X} V_{x_t}(x)$.
Furthermore, for either definition of the function $q_t$,
\begin{align*}
\left|f(\bar{x}_T,u_T) - \min_{x \in X} f(x,u^*)\right| ~\leq &~ \sum_{t=1}^T \theta_t q_t(x_t) - \inf_{x \in X} \sum_{t=1}^T \theta_t q_t(x)
 + |f(\bar{x}_T,u_T) - f(\bar{x}_T,u^*)| \\ & \qquad + D \sum_{t=1}^{T} \theta_t \| \grad_x f(x_t,u_t) - \grad_x f(x_t,u^*) \|_*.
\end{align*}
\end{lemma}
\begin{proof}
We will first consider the case when Assumption~\ref{ass:JEO-strongly-convex} holds and work with $q_t(x) := \la \grad_x f(x_t,u_t), x \ra + \alpha_{f,X} V_{x_t}(x)$.  
If Assumption~\ref{ass:JEO-strongly-convex} does not hold, the same proof applies with $\alpha_{f,X} = 0$. Assumption~\ref{ass:JEO-strongly-convex} implies that for any $x \in X$, $f(x_t,u^*) - f(x,u^*) \leq \la \grad_x f(x_t,u^*), x_t - x \ra - \alpha_{f,X} V_{x_t}(x)$. In addition, for any $t\in[T]$,
\begin{align*}
\la \grad_x f(x_t,u^*), x_t - x \ra &= \la \grad_x f(x_t,u_t), x_t - x \ra + \la \grad_x f(x_t,u^*) - \grad_x f(x_t,u_t), x_t - x \ra,\\
&\leq \la \grad_x f(x_t,u_t), x_t - x \ra + D \| \grad_x f(x_t,u_t) - \grad_x f(x_t,u^*) \|_*,
\end{align*}
where the inequality follows from Cauchy-Schwarz applied to $\la \grad_x f(x_t,u^*) - \grad_x f(x_t,u_t), x_t - x \ra$ and recognizing $\|x_t - x\| \leq D$ from Assumption~\ref{ass:JEO-basic-assumptions}. After subtracting $\alpha_{f,X} V_{x_t}(x)$ from both sides of this inequality for $t$, multiplying the resulting inequalities with $\theta_t$, summing them over $t\in[T]$ and using strong convexity of $f(\cdot,u^*)$, we arrive at:
\begin{align*}
f(\bar{x}_T,u^*) - f(x,u^*) \leq &\sum_{t=1}^T \theta_t (\la \grad_x f(x_t,u_t), x_t - x \ra - \alpha_{f,X} V_{x_t}(x)) \\ 
&\quad + D \sum_{t=1}^T \theta_t \| \grad_x f(x_t,u_t) - \grad_x f(x_t,u^*) \|_*.
\end{align*}
Then the first result follows from $\la \grad_x f(x_t,u_t), x_t - x \ra - \alpha_{f,X} V_{x_t}(x) = q_t(x_t) - q_t(x)$, and taking the maximum of both sides over $x \in X$. 
The last result follows from the triangle inequality
\[ \left|f(\bar{x}_T,u_T) - \min_{x \in X} f(x,u^*)\right| \leq |f(\bar{x}_T,u_T) - f(\bar{x}_T,u^*)| + \left|f(\bar{x}_T,u^*) - \min_{x \in X} f(x,u^*)\right|. \]
\end{proof}

The last result of Lemma~\ref{lemma:JEO-regret-f-strongly-convex} provides a bound on the gap between a computable quantity $f(\bar{x}_T,u_T)$ and the true optimum defined by the correct data $u^*$. This bound incurs additional penalty terms, $D \sum_{t=1}^T \theta_t \| \grad_x f(x_t,u_t) - \grad_x f(x_t,u^*) \|_*$ and $|f(\bar{x}_T,u_T) - f(\bar{x}_T,u^*)|$, which disappear when $u_t = u^*$ for all $t\in [T]$. Hence, these penalty terms can be interpreted as the `cost' of not working with the correct data $u^*$.

In order to ensure high quality solutions to the JEO problem, we need to bound the gap $\left|f(\bar{x}_T,u_T) - \min_{x \in X} f(x,u^*)\right|$, and by Lemma~\ref{lemma:JEO-regret-f-strongly-convex} this entails bounding three quantities: the regret term associated with the functions $q_t$ and the two penalty terms. We next demonstrate how the results from \cite{AhmadiShanbhag2014} on bounding the penalty terms can be recovered from our OCO based analysis. We work under the common assumption of \cite{AhmadiShanbhag2014} that $g$ is smooth and strongly convex, which assures the existence of algorithms with linear convergence $\|u_t-u^*\| = O(\beta^t)$ for our sequence $u_t$, and some mild Lipschitz continuity assumptions on $f(x,\cdot)$ and $\grad_x f(x,\cdot)$. Note that essentially the same results are achievable even if we assume $g$ is non-smooth and strongly convex. In such a case we can quarantine $\|u_t - u^*\| = O(1/t)$, and using this the modification of the other parts of Fact~\ref{fact:u-convergence} below with the replacement of little-$o$ notation with big-$O$ notation is immediate.
\begin{assumption}\label{ass:JEO-u-variation}\emph{}
\begin{itemize}
\item The function $g$ in \eqref{eqn:JEO-estimation-problem} is strongly convex and smooth in $u$.
\item There exists $G_{f,U} > 0$ such that for all $u,u' \in U$ and $x \in X$, it holds that $|f(x,u) - f(x,u')| \leq G_{f,U} \|u - u'\|$.
\item There exists $L_{f,U} > 0$ such that for all $u,u' \in U$ and $x \in X$, we have 
\[ \| \grad_x f(x,u) - \grad_x f(x,u') \|_* \leq L_{f,U} \|u - u'\|. \]
\end{itemize}
\end{assumption}
Under Assumption~\ref{ass:JEO-u-variation}, we can bound the two penalty terms in terms of the norms $\|u_t-u^*\|$ as:
\begin{align*}
|f(\bar{x}_T,u_T)-f(\bar{x}_T,u^*)| &\leq G_{f,U} \|u_T - u^*\|\\
D\, \sum_{t=1}^T \theta_t \| \grad_x f(x_t,u_t) - \grad_x f(x_t,u^*) \|_* &\leq D\, L_{f,U} \sum_{t=1}^T \theta_t \| u_t - u^* \|.
\end{align*}
Since we assume that $\|u_t-u^*\| = O(\beta^t)$, we can further bound the penalty terms using the following fact.
\begin{fact}\label{fact:u-convergence}
Consider a sequence $\{u_t\}_{t=1}^T$ such that $\|u_t- u^*\| = O(\beta^t)$ for some $0<\beta<1$. Then
\begin{enumerate}[(i)]
\item $\| u_T - u^* \| = o(1/T)$. 
\item For $\theta_t=1/T$, we have $\sum_{t=1}^T \theta_t  \| u_t - u^* \| = O(1/T) = o(1/\sqrt{T})$. 
\item For $\theta_t=\frac{2t}{T(T+1)}$, we have $\sum_{t=1}^T \theta_t  \| u_t - u^* \| = O(1/T^2) = o(1/T)$. 
\end{enumerate}
\end{fact}
\begin{proof}
Because $\|u_t - u^*\| \leq O(\beta^t) \leq o\left(\frac{1}{T}\right)$, item $(i)$ follows immediately. For item $(ii)$, when $\theta_t={1\over T}$, we note that
\[ \sum_{t=1}^T \theta_t \| u_t - u^*\| \leq \frac{1}{T}\ O\!\left(\sum_{t=1}^T \beta^t\right) = O\left( \frac{1}{T} \right).\]
For item $(iii)$, when  $\theta_t={2t \over{T(T+1)}}$, we observe that
\begin{align*}
\sum_{t=1}^T \theta_t \| u_t - u^*\| &\leq \frac{2}{T(T+1)}\ O\!\left(\sum_{t=1}^T t \beta^t\right)\\
&\leq\frac{2}{T(T+1)}\ O\!\left( \frac{\beta(1 - (T+1)\beta^T + T \beta^{T+1})}{(1-\beta)^2}\right)
= O\!\left( \frac{1}{T^2} \right) = o\!\left(\frac{1}{T}\right).
\end{align*}
\end{proof}

To complete our bound of the gap $|f(\bar{x}_T,u_T) - \min_{x \in X} f(x,u^*)|$, it remains to bound the weighted regret term associated with the functions $q_t$ in Lemma~\ref{lemma:JEO-regret-f-strongly-convex}. We can do so by using our results from Section~\ref{sec:weighted-regret-algorithms}. We summarize the cases when $f(\cdot,u)$ is not strongly convex in the following remark. Note that these cases are covered by \cite[Propositions 4 and 6]{AhmadiShanbhag2014}.
\begin{remark}\label{rem:JEO-f-convex-rate}
Suppose that Assumption~\ref{ass:JEO-basic-assumptions} holds, and that we are given a sequence $\{u_t\}_{t=1}^T$ of points from $U$. Given $x_t$, define $q_t(x) = \la \grad_x f(x_t,u_t),x \ra$. By applying Theorem~\ref{thm:OCO-non-smooth} appropriately with uniform weights $\theta_t = 1/T$, we obtain the regret bound 
\[\sum_{t=1}^T \theta_t q_t(x_t) - \inf_{x \in X} \sum_{t=1}^T \theta_t q_t(x) \leq \sqrt{\frac{2 \Omega G_{f,X}^2}{T}} = O\left(\frac{1}{\sqrt{T}}\right).\]
By Fact~\ref{fact:u-convergence}, in this case the penalty terms in Lemma~\ref{lemma:JEO-regret-f-strongly-convex} are asymptotically negligible $o(1/\sqrt{T})$ compared to the regret bound. This then recovers the overall convergence rate of $O(1/\sqrt{T})$ for solving JEO under the basic Assumption~\ref{ass:JEO-basic-assumptions}, see \cite[Proposition 6]{AhmadiShanbhag2014}.

If, in addition, Assumption~\ref{ass:JEO-smooth} holds, then by applying Theorem~\ref{thm:OCO-smooth} appropriately with uniform weights $\theta_t = 1/T$, the regret associated with $q_t$ is bounded by
\[\sum_{t=1}^T \theta_t q_t(x_t) - \inf_{x \in X} \sum_{t=1}^T \theta_t q_t(x) \leq \frac{L_{f,X} \Omega}{T} = O\left(\frac{1}{T}\right).\]
By Fact~\ref{fact:u-convergence}, the penalty terms in this case are asymptotically equivalent $O(1/T)$ to the regret bound. Hence, we recover the overall convergence rate of $O(1/T)$ for solving JEO under Assumptions~\ref{ass:JEO-basic-assumptions}~and~\ref{ass:JEO-smooth}, see \cite[Proposition 4]{AhmadiShanbhag2014}. 

Notably, these rates achieved in the JEO framework are the \emph{same rates} for FOMs for solving \eqref{eqn:JEO-problem} for the corresponding classes of functions $f$ when the correct data $u^*$ is available.
\epr
\end{remark}

We now study the case where $f$ is non-smooth and strongly convex; this case was not covered in \cite{AhmadiShanbhag2014}.
\begin{theorem}\label{thm:JEO-f-strongly-convex}
Suppose that Assumptions~\ref{ass:JEO-basic-assumptions} and~\ref{ass:JEO-strongly-convex} hold, and that we are given a sequence $\{u_t\}_{t=1}^T$ of points from $U$. Given $x_t \in X$, define $q_t(x) = \la \grad_x f(x_t,u_t), x \ra + \alpha_{f,X} V_{x_t}(x)$. Running Algorithm~\ref{alg:md} with $z_t = x_t$, $\xi_t = \theta_t \grad_x f(x_t,u_t)$, weights $\theta_t = \frac{2t}{T(T+1)}$ and step sizes $\gamma_t = \frac{2}{\alpha(t+1)}$ for $t \in [T]$ results in the bound
\[
\sum_{t=1}^T \theta_t q_t(x_t) - \inf_{x \in X} \sum_{t=1}^T \theta_t q_t(x) \leq \frac{2 G_{f,X}^2}{\alpha_{f,X} (T+1)} = O\left(\frac{1}{T}\right).
\]
Furthermore, suppose that Assumption~\ref{ass:JEO-u-variation} holds, and that $\|u_t - u^*\| = O(\beta^t)$. Define $\bar{x}_T = \sum_{t=1}^T \theta_t x_t$. Then
\[ \left|f(\bar{x}_T,u_T) - \min_{x \in X} f(x,u^*)\right| = O\left(\frac{1}{T}\right) + o\left(\frac{1}{T}\right).\]
\end{theorem}
\begin{proof}
Assumptions~\ref{ass:JEO-basic-assumptions}~and~\ref{ass:JEO-strongly-convex} ensure that the assumptions of Theorem~\ref{thm:OCO-strongly-convex} are met, which gives us the regret bound on $q_t$ (note also the equation in \eqref{eqn:weighted-regret-strongly-convex}). Then we use Lemma~\ref{lemma:JEO-regret-f-strongly-convex} to decompose the bound on $|f(\bar{x}_T,u_T) - \min_{x \in X} f(x,u^*)|$ into the regret term and the penalty terms. Also, from Assumption~\ref{ass:JEO-u-variation}  and  Fact~\ref{fact:u-convergence}, the penalty terms satisfy 
\begin{align*}
&|f(\bar{x}_T,u_T)-f(\bar{x}_T,u^*)| \leq G_{f,U} \|u_T-u^*\| = O(\beta^T) = o\left(\frac{1}{T}\right), \\
&D \sum_{t=1}^T \theta_t \| \grad_x f(x_t,u_t) - \grad_x f(x_t,u^*) \|_* \leq D\, L_{f,U} \sum_{t=1}^T \theta_t \| u_t - u^* \|= O\left(\frac{1}{T^2}\right) = o\left(\frac{1}{T}\right).  
\end{align*}
The result then follows.
\end{proof}

Notice that both penalty terms in Theorem~\ref{thm:JEO-f-strongly-convex} are $o(1/T)$, that is, asymptotically negligible compared to the $O(1/T)$ error. Thus, when the data generation process \eqref{eqn:JEO-estimation-problem} involves minimizing a smooth and strongly convex function $g$, the simultaneous JEO approach in Theorem~\ref{thm:JEO-f-strongly-convex} achieves the optimal offline rate of $O(1/T)$ for minimizing non-smooth strongly convex functions \cite[Theorem 3.13]{Bubeck2015}, plus some asymptotically negligible $o(1/T)$ penalty for not using the correct data. The analysis presented above depends crucially on the regret bound for the sequence of functions $\{q_t(x) = \la \grad_x f(x_t,u_t),x \ra + \alpha_{f,X} V_{x_t}(x)\}_{t=1}^T$. Therefore, by Remark~\ref{rem:stronglyConvDiscussion}, if we restricted ourselves to standard regret, we would only be able to get a bound of $O(\log(T)/T)$. Thus, our developments and analysis of weighted regret are fundamental in achieving the rate $O(1/T)$.

\section{Conclusion}\label{sec:Conclusions}
In this paper, we examine iterative solution techniques for RO and JEO through the lens of OCO and study their structure-based acceleration. For this purpose, we advance the line of research in OCO by introducing the concepts of weighted regret, online SP problems, and studying their implications when the decisions are restricted to be made in either non-anticipatory or 1-lookahead fashion. Our analyses demonstrate that when structural information such as smoothness or strong convexity of the loss functions is present, the additional flexibility introduced to the OCO framework by allowing weighted regret and/or 1-lookahead decisions can lead to significant improvements in the convergence rates. 
These then have immediate consequences on the convergence rates of iterative methods for solving RO problems studied in \cite{BenTalHazan2015,Ho-NguyenKK2016RO}; in particular Theorem~\ref{thm:OCO-strongly-convex} helps in partially resolving an open question from \cite{BenTalHazan2015} for the lower bound on the number of iterations/calls needed in these iterative frameworks for RO. Moreover, our results also have immediate application in the simultaneous JEO approach studied in \cite{JiangShanbhag2013,JiangShanbhag2014,AhmadiShanbhag2014}.   We establish that, in certain cases, our convergence rates for JEO, despite working with only estimates $u_t$ approximating the correct data   $u^*$, match the optimum lower bounds established for offline FOMs solving problems supplied with the correct data $u^*$.

There are a number of compelling avenues for future research. 
We believe our results may be further applicable to solve problems with uncertain data in the same spirit of Sections~\ref{sec:RO-app} and \ref{sec:JEO-app} and may open up possibilities for more principled solution  approaches in other application domains.   
An important extension of particular interest is the case where the learning problem \eqref{eqn:JEO-estimation-problem} in JEO is no longer static, but it dynamically evolves over time. 
Lower complexity bounds have been previously established for offline FOMs for problems over simple domains as well as some specific OCO problems. Nevertheless, the flexibilities we have introduced here point out that some of these lower bounds are no longer valid in the new setups (see Remarks~\ref{rem:stronglyConvDiscussion}~and~\ref{rem:smoothLBdiscuss}). Thus, establishing lower bounds matching our weighted regret (online SP gap) bounds in these setups are of interest. In particular, establishing the tightness of $O(1/T)$ bounds for weighted regret of strongly convex loss functions has a major consequence in determining the worst-case complexity of iterative approaches for solving RO problems.  
From a practical perspective,  
in certain applications and/or OCO contexts, it may be reasonable to assume that the players are not presented with exact feedback in the form of gradient/subgradient information but with only their unbiased estimates. Then deriving online stochastic iterative algorithms and studying the impact of several choices such as weighted regret, lookahead decisions, etc., on their behavior is  
of practical and theoretical interest. 
In this paper, we have worked under the assumption that our domain is convex; however both RO and JEO have many applications with nonconvex domains, e.g., involving discrete decision variables. A few online learning algorithms do not rely on such convexity assumption. It is appealing to study the implications of weighted regret and lookahead decisions for such algorithms and their potential use in solving online SP problems as well.

\section*{Acknowledgments}
The authors wish to thank the review team for their constructive feedback that improved the presentation of the material in this paper. 
This research is supported in part by NSF grant CMMI 1454548.

\bibliographystyle{abbrv}

\end{document}